\documentclass[11pt,reqno]{amsart}
\usepackage{amssymb,amsmath,amsthm}
\usepackage{graphicx}
\usepackage{subfigure}
\usepackage{color}
\newtheorem{theorem}{Theorem}[section]
\newtheorem{lemma}[theorem]{Lemma}
\newtheorem{proposition}[theorem]{Proposition}

\theoremstyle{definition}

\theoremstyle{remark}
\newtheorem{remark}[theorem]{Remark}

\makeatletter
\@namedef{subjclassname@2020}{2020 Mathematics Subject Classification}
\makeatother

\numberwithin{equation}{section}


\date{\today}

\begin{document}

\title[Perturbation of the spectra for asymptotically constant differential operators]
{Perturbation of the spectra for asymptotically constant differential operators and applications}

\maketitle

\vskip 1pt
\centerline{\scshape Shuang Chen}
\medskip
{\footnotesize
 \centerline{School of Mathematics and Statistics, Central China Normal University}
 \centerline{Wuhan, Hubei 430079, China}
 \centerline{Hubei Key Laboratory of Mathematical Sciences, Central China Normal University}
 \centerline{Wuhan, Hubei 430079, China}
   \centerline{{\rm{Email}: schen@ccnu.edu.cn}}
} 

\medskip

\centerline{\scshape Jinqiao Duan\footnote{The corresponding author}}
\medskip

{\footnotesize
 \centerline{Department of Applied Mathematics, Illinois Institute of Technology}
   \centerline{Chicago, IL 60616, USA}
 \centerline{Department of physics, Illinois Institute of Technology}
   \centerline{Chicago, IL 60616, USA}
    \centerline{{\rm{Email}: duan@iit.edu}}
} 

\medskip

\begin{abstract}

We study the spectra  for a class of differential operators with asymptotically constant coefficients.
These operators widely arise as the linearizations of nonlinear partial differential equations  about patterns or nonlinear waves.
We present a unified framework to prove the perturbation results on the related spectra.
The proof is based on exponential dichotomies and the Brouwer degree theory.
As applications, we employ the developed theory to study
the stability of quasi-periodic solutions of the Ginzburg-Landau equation,
fold-Hopf bifurcating periodic solutions of reaction-diffusion systems coupled with ordinary differential equations,
and periodic annulus  of the hyperbolic Burgers-Fisher model.

\end{abstract}

\section{Introduction}

The study of many structures,  such as nonlinear waves and patterns,
plays a pivotal role in understanding the underlying mechanisms of physical, biological and chemical phenomena.
One of the important problems is to study the dynamics of these permanent structures.
The key step towards this is to analyze the spectra of the related linearized operators.
This has attracted attention of many authors  and led to many interesting developments
in the past two decades.
We refer to \cite{Kapitula-Promislow-13,Sandstede-02} and references therein.

Our goal in the present paper is to study the spectra of
linear differential operators with asymptotically constant coefficients
arising in analyzing the local stability of patterns.
More precisely, we  consider a second-order operator of the form
\begin{eqnarray} \label{df-L}
\mathcal{L}p:=D\partial^{2}_{\xi}p+a_{1}(\xi,\epsilon)\partial_{\xi}p+a_{0}(\xi,\epsilon)p,
\end{eqnarray}
where $p(\xi)\in\mathbb{C}^{n}$, $\epsilon\in\mathbb{R}^{n_{1}}$ is a parameter vector, $\xi\in \mathbb{R}$ is the spatial variable,
$D$ is a $n\times n$ matrix, and  the coefficients $a_{j}(\xi,\epsilon)$ are
real matrix functions and smooth in $\xi$.

It is well known that many nonlinear evolutionary systems,
such as reaction-diffusion systems, hyperbolic systems of conservation laws and
Hamiltonian partial differential equations, admit
small-amplitude periodic traveling waves (wave trains) emerging from  homogeneous rest states,
and quasi-periodic patterns bifurcating from periodic waves.
See, for instance,
\cite{Alvarez-Plaza-21,Barker-Jung-Zumbrun-18,Chen-Duan-21,Doelman-Cardner-Jones-95,Duan-Holme-95,Dunbar-86,
Huang-Lu-Ruan-03,Kollar-etal-19,Sherratt-2008,Tsai-etal-2012}.
Typical examples are listed as follows:

{\bf Example 1.}  The first example is the real Ginzburg-Landau equation
\begin{eqnarray}\label{GLE}
u_{t}=(1-|u|^{2})u+\partial^{2}_{\xi}u.
\end{eqnarray}
This equation admits quasi-periodic steady  solutions of the form
$u(\xi)=p(\xi)e^{{\bf i}\theta(\xi)}$ for $\xi\in\mathbb{R}$,
where both $p(\xi)$ and $\partial_{\xi}\theta(\xi)$ are periodic.
See \cite{Doelman-Cardner-Jones-95,Newell-Whitehead-69} for instance.
By a direct computation, we see that $(p,\theta)$ satisfies
\begin{eqnarray}\label{GLE-p-theta}
p^{2}\theta'=\omega,\ \ \ \ p''+f(p)=0,
\end{eqnarray}
where $\omega\in\mathbb{R}$ and $f(p)=-\omega^{2}/p^{3}+p-p^{3}$.
For each $\omega$  with $|\omega|<\sqrt{4/27}$,
the function $f$ has two positive zeros $p=p_{1}$ and $p=p_{2}$ with $0<p_{1}<p_{2}$,
and the equivalent system of the second equation in \eqref{GLE-p-theta}
has a family of periodic solutions surrounding $(p_{1},0)$.
Let $T(p)$ denote the periods of these periodic solutions starting from $(p,0)$.
We observe that $T(p)$ has the limit
\begin{eqnarray*}
\lim_{p\to p_{1}}T(p)=2\pi/\sqrt{f'(p_{1})}.
\end{eqnarray*}
Let $\alpha(\xi,t)+{\bf i}\beta(\xi,t)=u(\xi,t)\exp(-{\bf i}\theta(\xi))$
for $\alpha$ and $\beta$ in $\mathbb{R}$.
Then $(\alpha,\beta)$ satisfies a two-component reaction-diffusion equation
(see (2.3) and (2.4) in \cite[p.505]{Doelman-Cardner-Jones-95}),
which has a periodic steady solution $(p(\xi),0)$.
The corresponding  linearization is in the form  \eqref{df-L}.
See (2.5) and (2.6) in \cite[p.505]{Doelman-Cardner-Jones-95}.
Doelman, Gardner and Jones \cite{Doelman-Cardner-Jones-95} once applied
the geometric and topological methods developed by Gardner in \cite{Gardner-93}
to obtain the instability of quasi-periodic solutions $u(\xi)=p(\xi)e^{{\bf i}\theta(\xi)}$
near the periodic solution $u_{0}(\xi)=p_{1}e^{{\bf i}k\xi}$,
where $k^{2}+p_{1}^{2}=1$ and $kp_{1}^{2}=\omega$.
Here, our criteria will give a simpler way to determine the stability of quasi-periodic solutions
near $u_{0}(\xi)$.

{\bf Example 2.} The second example is a system of one reaction-diffusion equation coupled with
one ordinary differential equation (ODE)
\begin{eqnarray} \label{PDE-ODE}
\begin{aligned}
\frac{\partial u}{\partial t} &=u_{xx}+f(u,w,\alpha),\\
\frac{\partial w}{\partial t} &=g(u,w,\alpha),
\end{aligned}
\end{eqnarray}
where $x\in\mathbb{R}$, $\alpha\in\mathbb{R}^{m}$ for $m\geq 1$, and $f$ and $g$ are sufficiently smooth.
This model includes the  FitzHugh-Nagumo system \cite{FitzHugh-60,Nagumo-62},
caricature calcium models \cite{Tsai-etal-2012},
and population dynamics models  \cite{Zhang-etal-17} as special cases.
If  $(u(x,t),w(x,t))=(\phi(x+ct),\psi(x+ct))$ is a traveling wave solution of \eqref{PDE-ODE},
then $(\phi(\cdot),\phi_{\xi}(\cdot), \psi(\cdot))$ satisfies  a three-dimensional ordinary differential system of the form
\begin{eqnarray} \label{3D-SYSTEM}
\begin{aligned}
\frac{d u}{d \xi} &= \dot u  = v,
\\
\frac{d v}{d \xi} &= \dot v  = c v-f(u,w,\alpha),
\\
\frac{d w}{d \xi} &= \dot w  = \frac{1}{c} g(u,w,\alpha).
\end{aligned}
\end{eqnarray}
Assume that this three-dimensional system has a fold-Hopf equilibrium for $(\alpha,c)=(\alpha_{0},c_{0})$,
whose Jacobian matrix has one zero eigenvalue and two purely imaginary eigenvalues $\pm{\bf i}\mu_{0}$ with $\mu_{0}>0$.
Then there exist two functions $\alpha=\alpha_{0}+O(\epsilon)$ and $c=c_{0}+O(\epsilon)$,
for sufficiently small $\epsilon>0$,
such that a small-amplitude periodic solution $(\phi_{\epsilon},\psi_{\epsilon})$ bifurcates from the fold-Hopf equilibrium
and satisfies
\[
|(\phi_{\epsilon},\psi_{\epsilon})|=O(\epsilon), \ \ \ T_{\epsilon}=2\pi/\mu_{0}+O(\epsilon),
\]
for sufficiently small $\epsilon$. See \cite[Theorem 2.3]{Chen-Duan-21}.
By applying a perturbation method, our recent work \cite{Chen-Duan-21} proved that
the linearization of \eqref{PDE-ODE} about this perturbed periodic wave is spectrally unstable.
However, we only considered the perturbation of  a special eigenvalue,
and did not discuss the perturbation of the spectral curves.

{\bf Example 3.} The third example is the hyperbolic Burgers-Fisher model
\begin{eqnarray} \label{H-model}
\begin{aligned}
u_{t}+v_{x} &=g(u),\\
\alpha v_{t}+u_{x} &=f(u)-v,
\end{aligned}
\end{eqnarray}
where $(u,v)\in\mathbb{R}^{2}$, $x\in\mathbb{R}$ and $t>0$.
Here the reaction is of the logistic growth:
\[
g(u)=u(1-u),
\]
and the nonlinear advection flux is the Burgers' flux
\[
f(u)=\frac{1}{2}u^{2}.
\]
The hyperbolic Burgers-Fisher model \eqref{H-model} is based on
the universal balance law and the constitutive law of Cattaneo-Maxwell type \cite{Cattaneo-49},
which are respectively governed by the first and the second equations in \eqref{H-model}.
We also refer to 
\cite{Alvarez-Murillo-Plaza-22,Fisher-37,Holmes-93,Joseph-Preziosi-89,Liu-87} for more information about this model.

If $(u(x,t),v(x,t))=(\phi(x+ct),\psi(x+ct))$ is a periodic traveling wave of
the hyperbolic Burgers-Fisher model \eqref{H-model},
then the profile $(\phi(\xi),\psi(\xi))$ satisfies the planar dynamical system
\begin{eqnarray} \label{H-2D-model}
\begin{aligned}
cu'+v' &=u(1-u),\\
u'+\alpha c v' &=\frac{1}{2}u^{2}-v,
\end{aligned}
\end{eqnarray}
where $'$ denotes the derivative with respect to $\xi$.

Assume that the wave speed $c$ satisfies $\alpha c^{2}<1$.
This assumption is the so-called  subcharacteristic condition.
Under this condition, system \eqref{H-2D-model} has exactly two equilibria:
$E_{1}:=(0,0)$ and $E_{2}:=(1,1/2)$.
When $\alpha=1$ and $c=-1/2$,
system \eqref{H-2D-model} is reduced to
\begin{eqnarray} \label{H-2D-model-1}
\begin{aligned}
u' &=\frac{2}{3}u-\frac{4}{3}v,\\
v' &=-u^{2}+\frac{4}{3}u-\frac{2}{3}v.
\end{aligned}
\end{eqnarray}
The system \eqref{H-2D-model-1} is an integrable system with the first integral
\[
H(u,v):=-\frac{2}{3}u^{2}+\frac{1}{3}u^{3}+\frac{2}{3}uv-\frac{2}{3}v^{2}, \ \ \ \ (u,v)\in\mathbb{R}^{2},
\]
and $E_{1}$ and $E_{2}$ are a center and a saddle, respectively.
We see that system \eqref{H-2D-model-1} has a periodic annulus, that is,
a family of periodic orbits  $(\phi_{h},\psi_{h}):=\left\{(u,v)\in\mathbb{R}^{2}: H(u,v)=h\right\}$ for $h\in(-1/6,0)$.
We are interested in the spectral stability of periodic solutions in the annulus
with respect to \eqref{H-model}.

Inspired by these interesting examples,
throughout this paper we assume that
the functions $a_{j}$ in \eqref{df-L} satisfy the following hypotheses:

\begin{enumerate}
\item[{\bf (H1)}]
For some small $\varepsilon_{0}>0$ and each $\epsilon$ with $0<|\epsilon|< \varepsilon_{0}$,
there exists a constant $T_{\epsilon}>0$ such that
\[
a_{j}(\xi+T_{\epsilon},\epsilon)=a_{j}(\xi,\epsilon),\ \ \ \ \ j=0,1, \ \ \ \xi\in\mathbb{R},
\]
where $T_{\epsilon}\to T_{0}>0$ as $\epsilon\to 0$.

\item[{\bf (H2)}]
There exist two constant matrices $a_{j}^{0}$ ($j=0,1$) such that
the periodic functions $a_{j}$ satisfy

\begin{eqnarray*}
|a_{j}(\xi,\epsilon)-a_{j}^{0}|=O(\epsilon),\ \  \ \ \ j=0,1,
\end{eqnarray*}
for sufficiently small $|\epsilon|$, uniformly in the variable $\xi$.
Then the coefficients $a_{j}(\xi,\epsilon)$ are said to be asymptotically constant with respect to $\epsilon$.
\end{enumerate}
Actually,
regarding the quasi-periodic patterns near a periodic wave in Example 1, the fold-Hopf bifurcating periodic waves in Example 2,
and the periodic waves near a constant steady state in Example 3,
the related linearized operators take the form \eqref{df-L}
whose coefficients have the properties in {\bf (H1)} and {\bf (H2)}
(see more details in Section \ref{sec:app}).

It is well known that the spectrum of the linear operator $\mathcal{L}$ considered on the whole line
has purely essential spectrum, and is determined by the Floquet multipliers of the related eigenvalue problem.
Those properties always cause a big obstacle in studying the spectrum of periodic problems \cite{Gardner-93,Kapitula-Promislow-13}.
The perturbation method for linear operators has been successfully applied to
study the spectral stability of small-amplitude periodic traveling waves,
whose linearizations are in the form \eqref{df-L} with {\bf (H1)} and {\bf (H2)}.
See, for instance, \cite{Alvarez-Plaza-21,Alvarez-Murillo-Plaza-22,Chen-Duan-21,Kollar-etal-19}
and the references therein.
In the view of perturbation theory,
it is possible to analyze the spectrum of the linear operator $\mathcal{L}$ with {\bf (H1)} and {\bf (H2)}
by the limiting case of $\mathcal{L}$, i.e.,
the limiting operator $$\mathcal{L}_{0}:=D\partial^{2}_{\xi}+a_{1}^{0}\partial_{\xi}+a_{0}^{0},$$
obtained by letting $\epsilon\to 0$ in \eqref{df-L}.
Since $\mathcal{L}_{0}$ is a linear differential operator with constant coefficients,
its spectral analysis can be accomplished by a direct computation.
This advantage is helpful to analyze the spectrum of $\mathcal{L}$ by perturbation techniques.

Following this idea,
\cite{Alvarez-Plaza-21,Alvarez-Murillo-Plaza-22} and our recent work \cite{Chen-Duan-21} applied  Kato's theory \cite{Kato-95}
to prove the spectral instability of Hopf bifurcating periodic waves for some hyperbolic models
and fold-Hopf bifurcating periodic waves for model \eqref{PDE-ODE}, respectively.
While Kato's theory is powerful,
complicated computations are always involved in determining whether
the linearized operators about small-amplitude perturbed periodic waves are
relatively bounded perturbations of the related  limiting operators.

In this paper, we shall present a new method which allows us to recover the previous results
in \cite{Alvarez-Plaza-21,Alvarez-Murillo-Plaza-22,Chen-Duan-21}. 
The proof for our criteria is based on exponential dichotomies \cite{Coppel-78},
combined with the Brouwer degree theory \cite{Browder-83}.
With our results, it becomes easier to check the instability of perturbed periodic waves
arising from Hopf bifurcation and fold-Hopf bifurcation in dissipative systems.
As applications, we apply our stability criteria to study
the stability of quasi-periodic solutions of the  Ginzburg-Landau equation \eqref{GLE},
fold-Hopf bifurcating periodic solutions of reaction-diffusion systems coupled with ODEs \eqref{PDE-ODE},
and periodic annulus associated with hyperbolic Burgers-Fisher model \eqref{H-model}.

We are also interested in the Floquet spectrum curves of the linear operator $\mathcal{L}$.
Although  the periodic Evans function method is applicable to locate them,
it is still difficult to deal with this problem.
As we see, the spectrum of $\mathcal{L}_{0}$ consists of several continuous curves
(called its {\it spectral curves}) if $\mathcal{L}_{0}$ is posed on $L^{2}(\mathbb{R})$
(see Section \ref{sec-Pro-limt}).
Intuitively, it is possible to study the Floquet spectrum curves of $\mathcal{L}$ by its limiting operator $\mathcal{L}_{0}$.
Under certain conditions,
we prove that some connected closed subsets of the spectral curves of $\mathcal{L}_{0}$
are perturbed to continuous closed curves in the spectrum of $\mathcal{L}$ (see (ii) of Theorem \ref{thm-1}),
yet without full description of the Floquet spectrum curves.
It is worth mentioning that
\cite{Bronski-Johnson-10} once derived an index to determine the local structure of the Floquet spectrum curves in the vicinity of the origin
for a generalized Korteweg-de Vries equation,
and then \cite{Johnson-Zumbrun-10} further connected the formal Whitham modulation theory with
the analysis of these Floquet spectrum curves.
We hope that our results together with the developed theory in  \cite{Bronski-Johnson-10,Johnson-Zumbrun-10} could
give more information on the Floquet spectrum curves of $\mathcal{L}$ arising from concrete practical problems.

This paper is organized as follows.
We first review exponential dichotomies and the Brouwer degree theory in Section \ref{sec-ED}.
Then we examine the properties of the limiting operators in three different cases in Section \ref{sec-Pro-limt}.
We present a detailed study of the perturbation of the related spectrum  in Section \ref{sec-pert}.
Finally, we apply the obtained results to study the local stability of patterns arising in the preceding three examples.

\section{Exponential dichotomies and Brouwer degree theory}
\label{sec-ED}

In this section,
we introduce some  basic results on exponential dichotomies for linear nonautonomous systems \cite{Coppel-78}
and the Brouwer degree theory \cite{Browder-83}.
Based on these two notions, we will obtain the spectral perturbation for
the linear operator $\mathcal{L}$ defined by \eqref{df-L}.

\subsection{Exponential dichotomies}
Exponential dichotomy describes the hyperbolicity of dynamical systems generated by linear maps or differential systems.
It plays a useful role in the study of the local stability of nonlinear waves \cite{Coppel-78,Kenneth-84,Kenneth-88,Sandstede-02}.
We introduce  two crucial properties of exponential dichotomies for  linear nonautonomous systems,
that is, roughness and admissibility of the exponential dichotomies.

Consider a linear nonautonomous system
\begin{eqnarray} \label{ODE}
X'=A(t)X,
\end{eqnarray}
where $ t\in\mathbb{R}$, $X\in\mathbb{C}^{n}$,
and the mapping $t\to A(t)\in\mathbb{C}^{n\times n}$ is continuous.
Let $\Phi(t)$ denote the principal fundamental matrix solution of the linear system \eqref{ODE}.
This system is said to admit an {\it  exponential dichotomy} on $\mathbb{R}$
if there exists a projection $P$ on $\mathbb{C}^{n}$, and two constants $K\geq 1$ and $\alpha>0$ such that
\begin{equation}
\begin{array}{cc}
|\Phi(t)P\Phi^{-1}(s)|\leq Ke^{-\alpha(t-s)}, &  t\geq s,\\
|\Phi(t)(I-P)\Phi^{-1}(s)|\leq Ke^{-\alpha(s-t)}, &  t\leq s,
\end{array}
\end{equation}
where $K$ and $\alpha$ are called the {\it bound} and the {\it exponent} of this dichotomies, respectively.
In particular,
if $A(t)$ is a constant matrix $A$ for every $t\in\mathbb{R}$,
then \eqref{ODE} admits an exponential dichotomy if and only if
the constant matrix $A$ has no eigenvalues with zero real parts.
We refer the readers to \cite{Coppel-78,Henry-81,Kuehn-19,Massera-66}
for more details on exponential dichotomies of linear  nonautonomous systems.

One of the most important properties of exponential dichotomies is the persistence of exponential dichotomies under
``small" perturbations in some sense.
This property is called the {\it roughness} of exponential dichotomies.
Consider a perturbed system of \eqref{ODE} in the form
\begin{eqnarray} \label{ODE-pert}
X'=(A(t)+B(t))X,
\end{eqnarray}
where the mapping $t\to B(t)\in\mathbb{C}^{n\times n}$ is continuous.
In order to study the spectrum of the linear operator $\mathcal{L}$ with periodic coefficients,
we further assume that the matrix functions $A(t)$ and $B(t)$ have period $T$,
that is, $A(t+T)=A(t)$  and $B(t+T)=B(t)$ for each $t\in\mathbb{R}$.

By the {\it Floquet Theorem} \cite[Theorem 7.1, p.118]{Hale-80},
there exists a constant matrix $Q$ and a periodic matrix function $P(t)$ with $P(t+T)=P(t)$
such that  the principal fundamental matrix solution $\Phi(t)$ of the linear system \eqref{ODE}
has the form
\begin{eqnarray}\label{Floq}
\Phi(t)=P(t)e^{Qt},\ \ \ \ \ t\in\mathbb{R}.
\end{eqnarray}
We call the matrix $e^{QT}$ the {\it monodromy operator} of \eqref{ODE},
and its eigenvalues the {\it Floquet multipliers}.
The eigenvalues of the matrix $Q$ are said to be the {\it Floquet exponents} of \eqref{ODE}.
We next use the Floquet exponents to characterize the exponential dichotomies of linear periodic systems,
and further show the roughness of these constants under small perturbations.

\begin{lemma}\label{lm-Period}
Assume that the matrix functions $A(t)$ and $B(t)$ in \eqref{ODE-pert} have period $T$.
Let  $\alpha_{j}+{\bf i}\beta_{j}$ and $\tilde{\alpha}_{j}+{\bf i}\tilde{\beta}_{j}$, $j=1,2,...,n$, denote
all  Floquet exponents of \eqref{ODE} and \eqref{ODE-pert}, respectively.
Then the following properties are satisfied:
\begin{enumerate}
\item[(i)] If all Floquet exponents have nonzero real part,
that is, $\alpha_{j}\neq 0$ for $j=1,2,...,n$,
then \eqref{ODE} admits an exponential dichotomy.

\item[(ii)] Assume that $|\alpha_{j}|+|\beta_{j}|\neq0$ for $j=1,2,...,n$.
Then there exists a sufficiently small $\delta_{0}>0$ such that
for each $\delta\in [0,\delta_{0}]$, $\tilde{\alpha}_{j}+{\bf i}\tilde{\beta}_{j}$
satisfy $|\tilde{\alpha}_{j}|+|\tilde{\beta}_{j}|\neq0$ for $j=1,2,...,n$,
and $(\tilde{\alpha}_{j},\tilde{\beta}_{j})\to (\alpha_{j},\beta_{j})$ as $\delta\to 0$.
\end{enumerate}
\end{lemma}
\begin{proof}
The result in (i) is proved by \eqref{Floq}.
The proof of (ii) is standard in the literature (see, for instance, \cite{Coppel-78}).
\end{proof}

In order to analyze the spectrum of the second-order operator $\mathcal{L}$ in \eqref{df-L},
we need another property of exponential dichotomies.

\begin{lemma} \cite[p.73]{Kapitula-Promislow-13}
\label{lm-admis}
{\rm (Admissibility of exponential dichotomies)}
Suppose that the linear system \eqref{ODE} admits an exponential dichotomy.
Then for each given $f\in L^{2}(\mathbb{R})$,
the inhomogeneous linear  system
\begin{eqnarray}\label{eq:linear-inhomo}
X'=A(t)X+f(t), \ \ \ \ \ t\in \mathbb{R},
\end{eqnarray}
has a unique solution $X\in L^{2}(\mathbb{R})$  which satisfies
 the estimate $\|X\|_{2}\leq K \|f\|_{2}$ for some constant $K>0$ independent of $f$.
\end{lemma}

This lemma indicates the admissibility of the pair $(L^{2}(\mathbb{R}), L^{2}(\mathbb{R}))$,
i.e., for every ``text function" $f$ in $L^{2}(\mathbb{R})$,
the inhomogeneous linear  system \eqref{eq:linear-inhomo} has a unique solution in $L^{2}(\mathbb{R})$.
This property is closely related to the Fredholm properties of the linearizations about nonlinear waves.
We also refer the readers to \cite{Coppel-78,Massera-66} for
the admissibility of other pairs of Banach spaces.

\subsection{Brouwer degree theory}

We see that the spectrum of the linear operator $\mathcal{L}$
can be detected by solving the zeros of an analytic function.
For this reason, we present basic results on the Brouwer degree theory.
This notion plays a useful  role in proving the existence of  zeros in bounded sets for analytic functions.

\begin{lemma}\cite[Theorem 1, p.5]{Browder-83}\label{lm-BDT-1}
Let $U$ be a bounded open set of $\mathbb{R}^{d}$ and its closure and boundary
are  denoted by $\bar{U}$ and $\partial U$, respectively.
Consider a continuous map $f: \bar{U}\to \mathbb{R}^{d}$,
and a point $z_{0}$ in $\mathbb{R}^{d}$ such that $z_{0}$ does not lie in $f(\partial U)$.
Then to each such triple $(f,U,z_{0})$, there corresponds an integer $d(f,U,z_{0})$
having the following properties:
\begin{enumerate}
\item[(i)] If $d(f,U,z_{0})\neq 0$, then $z_{0}\in f(U)$.
If $f$ is the identity map on $\mathbb{R}^{d}$,
then for each $z_{0}\in U$,
the restriction $f|_{U}$ of $f$ to $U$ satisfies
\begin{eqnarray*}
d(f|_{U},U,z_{0})=+1.
\end{eqnarray*}

\item[(ii)]
{\rm (Additivity)} If $U_{1}$ and $U_{2}$ are a pair of disjoint open subsets of $U$ such that
$z_{0}\notin f(\bar{U}/(U_{1}\cap U_{2}))$,
then $$d(f,U,z_{0})=d(f,U_{1},z_{0})+d(f,U_{2},z_{0}).$$

\item[(iii)] {\rm (Invariance under homotopy)}
Consider a continuous homotopy
$\{f_{t}: 0\leq t\leq 1\}$ of maps of $\bar{U}$ into $\mathbb{R}^{d}$.
Let $\{z_{t}: 0\leq t\leq 1\}$ be a continuous curve in $\mathbb{R}^{d}$ such that $z_{t}\notin f_{t}(\partial U)$
for each $t\in [0,1]$. Then $d(f_{t},U,z_{t})$ is constant for $t\in [0,1]$.
\end{enumerate}
\end{lemma}

\begin{lemma}  \cite[Theorem 2, p.5]{Browder-83}\label{lm-BDT-2}
The degree function $(f,U,z_{0})$ is uniquely determined by the three properties in Lemma \ref{lm-BDT-1}.
\end{lemma}

Now with the help of the above two  lemmas,
we consider the Brouwer degree for analytic functions from $\mathbb{C}$ to itself.
Let $f:\mathbb{C}\to \mathbb{C}$ be an analytic function.
Then all zeros of $f$ are isolated.
Let $U$ be a bounded open subset of $\mathbb{C}$
and $z_{0}$ be a point in $\mathbb{C}$ with $z_{0}\notin f(\partial U)$.
For the triple $(f,U,z_{0})$, we define $d_{B}(f,U,z_{0})$ by
\begin{eqnarray}\label{df-BDT}
d_{B}(f,U,z_{0})=\mbox{the number of zeros of $(f-z_{0})$ in $U$ (counted by multiplicity)}.
\end{eqnarray}

The next lemma is useful to detect the essential spectrum for small-amplitude periodic problems.
\begin{lemma}\label{lm-BDT-ANA}
The function $d_{B}(f,U,z_{0})$ defined by \eqref{df-BDT} is the Brouwer degree of $f$ at $z_{0}$.
\end{lemma}

Before giving the proof for Lemma \ref{lm-BDT-ANA},
we first introduce an important result on continuity of the roots of an equation as a function of parameters.
\renewcommand\thelemma{A}
\begin{lemma}\cite[Theorem 9.17.4, p.248]{Dieudonne-69}\label{lm-app-1}
Let $V$ be an open subset of $\mathbb{C}$, $B$ be a metric space, and $f$ be a continuous complex valued function
in $V\times B$ such that for each $s$ in $B$, the function $f(\cdot,s)$ is analytic in $V$.
Let $U$ be an open subset of $V$, whose closure $\bar{U}$ is compact and contained in $V$,
and let $s_{0}\in B$ satisfy $0\notin f(\partial U,s_{0})$.
Then there exists an open neighborhood $\mathcal{U}(s_{0})$ of $s_{0}$ in $B$ such that
for each $s\in \mathcal{U}(s_{0})$, the following properties hold:
\begin{enumerate}
\item[(i)]  The function $f(\cdot,s)$ has no zeros on $\partial U$.

\item[(ii)] The number of the zeros of $f(\cdot,s)$ that belong to $U$ is independent of $s$.
\end{enumerate}
\end{lemma}

By the above lemma, we prove Lemma \ref{lm-BDT-ANA}.

{\bf Proof of Lemma \ref{lm-BDT-ANA}.}
For each $z\in\mathbb{C}$, we write $z={\rm Re}z+{\bf i}{\rm Im}z=:x_{1}+{\bf i}x_{2}$.
Then $f:\mathbb{C}\to \mathbb{C}$ can be seen as a map from $\mathbb{R}^{2}$ to itself.
It is easy to verify that $d_{B}(f,U,z_{0})$ satisfies  (i) in Lemma \ref{lm-BDT-1}.
By the {\it Argument Principle}, the function $d_{B}(f,U,z_{0})$ satisfies (ii) in Lemma \ref{lm-BDT-1}.

Now we prove that $d_{B}(f,U,z_{0})$ also satisfies  (iii) in Lemma \ref{lm-BDT-1}.
Let $\{f_{t}: 0\leq t\leq 1\}$ be a continous homotopy of maps of $\bar{U}$ into $\mathbb{C}$,
and $\{z_{t}: 0\leq t\leq 1\}$ be a continuous curve in $\mathbb{C}$ such that $z_{t}\notin f_{t}(\partial U)$
for each $t\in [0,1]$.
We define a map $\tilde{f}:\mathbb{C} \times [0,1] \to \mathbb{C}$ by
\[
\tilde{f}(\lambda,t)=f_{t}(\lambda)-z_{t}.
\]
Then $0\notin \tilde{f}(\partial U,t)$ for each $0\leq t\leq 1$.
By Lemma \ref{lm-app-1} and the {\it Finite Covering Theorem},
we have that the number of zeros of $\tilde{f}(\lambda,t)$ in the set $U$ is a constant for each $t\in[0,1]$.
So the function satisfies the three properties stated in Lemma \ref{lm-BDT-1}.
Then the proof is finished by Lemma \ref{lm-BDT-2}.
\hfill $\Box$

\section{Properties of limiting operators}
\label{sec-Pro-limt}

In this section, we consider the properties of the limiting operators associated with $\mathcal{L}$.
Assume that the $n\times n$ matrix $D$ is given by
\begin{equation}\label{df-singular-D}
D=\left(
\begin{array}{cc}
I_{m} &  0_{m\times(n-m)}\\
0_{(n-m)\times m} &  0_{n-m}
\end{array}
\right), \ \ \ \ 1\leq m\leq n,
\end{equation}
where $I_{m}$ is the  $m\times m$ identity matrix,
and $0_{m\times(n-m)}$ and $0_{(n-m)\times m}$ denote the $m\times(n-m)$ and $(n-m)\times m$ zero matrices, respectively.
The whole discussion is divided in three different cases according to the term $D$:
{\bf (C1)} $D=I_{n}$, {\bf (C2)} $D={\rm diag}(I_{m}, O_{n-m})$ for $0<m<n$, and {\bf (C3)} $D=O_{n}$.

\subsection{The limiting operator in the first case}
\label{sec-case-1}

Suppose that $D=I_{n}$ in \eqref{df-L}.
Then the operator $\mathcal{L}$ is reduced to $\mathcal{L}^{1}$ of the form
\begin{eqnarray} \label{df-L-1}
\mathcal{L}^{1}p=\partial^{2}_{\xi}p+a_{1}(\xi,\epsilon)\partial_{\xi}p+a_{0}(\xi,\epsilon)p,
\end{eqnarray}
where
$$\mathcal{L}^{1}: H^{2}(\mathbb{R})\subset L^{2}(\mathbb{R}) \to L^{2}(\mathbb{R}).$$
The operator $\mathcal{L}^{1}$ is a closed operator on $L^{2}(\mathbb{R})$ with the domain $\mathcal{D}(\mathcal{L}^{1})=H^{2}(\mathbb{R})$.
We are interested in its eigenvalue problem
\[
\mathcal{L}^{1}p=\lambda p, \ \ \ \ \lambda\in \mathbb{C}.
\]

Let $X=(p,\partial_{\xi}p)^{t}$.
Throughout this paper,
the symbol $^t$ denotes the transpose of a matrix or a vector in the usual Euclidean inner product.
Then the eigenvalue problem $\mathcal{L}^{1}p=\lambda p$ is equivalent to
the following  first-order linear system:
\begin{eqnarray} \label{eq-period-1}
\partial_{\xi}X=A_{1}(\xi,\lambda,\epsilon)X,
\ \ \
A_{1}(\xi,\lambda,\epsilon):=\left(
\begin{array}{cc}
0 & I_{n} \\
\lambda I_{n}-a_{0}(\xi,\epsilon) & -a_{1}(\xi,\epsilon)
\end{array}
\right).
\end{eqnarray}
Note that the parameter $\lambda$ in \eqref{eq-period-1} is related to the spectrum of $\mathcal{L}^{1}$.
For notational convenience, we call $\lambda$ the {\it spectral parameter} of system \eqref{eq-period-1}.

Recall that the matrix functions $a_{j}(\xi,\epsilon)$ satisfy the hypothesis {\bf (H2)}.
In order to analyze the dynamics of system \eqref{eq-period-1},
we first consider the limit to system \eqref{eq-period-1} as $\epsilon\to 0$.
Letting $\epsilon\to 0$ in system \eqref{eq-period-1} yields
\begin{eqnarray} \label{eq-period-1-0}
\partial_{\xi}X=A_{1}(\lambda)X,
\ \ \
A_{1}(\lambda)=\left(
\begin{array}{cc}
0 & I_{n} \\
\lambda I_{n}-a_{0}^{0} & -a_{1}^{0}
\end{array}
\right).
\end{eqnarray}
This system with constant coefficients is said to be the {\it limiting system} of \eqref{eq-period-1}.
We also obtain that this system with constant coefficients is the extended system of
the eigenvalue problem $\mathcal{L}^{1}_{0}p=\lambda p$,
where $\mathcal{L}^{1}_{0}: H^{2}(\mathbb{R})\subset L^{2}(\mathbb{R}) \to L^{2}(\mathbb{R})$ is in the form
\begin{eqnarray*}
\mathcal{L}^{1}_{0}p=\partial^{2}_{\xi}p+a_{1}^{0}\partial_{\xi}p+a_{0}^{0}p, \ \ \ \ \ p\in H^{2}(\mathbb{R}).
\end{eqnarray*}
So we also call $\mathcal{L}^{1}_{0}$ the {\it limiting operator} of $\mathcal{L}^{1}$.

Next we give the conditions under which system \eqref{eq-period-1-0} admits an exponential dichotomy.
\begin{lemma} \label{lm-ED-constant-1}
For a given $\lambda\in\mathbb{C}$, system \eqref{eq-period-1-0}  admits an exponential dichotomy
if and only if
\begin{eqnarray}\label{ED-cond}
\Delta_{1}(\lambda,\mu):={\rm det}\left(\lambda I_{n}-(a_{0}^{0}-\mu^{2}I_{n}+{\bf i}a_{1}^{0}\mu)\right)\neq 0,
\ \ \ \ \mu\in\mathbb{R}.
\end{eqnarray}
\end{lemma}
\begin{proof}
It is clear that system \eqref{eq-period-1-0}  admits an exponential dichotomy
if and only if $A_{1}(\lambda)$ has no eigenvalues lying on the imaginary axis,
which is equivalent to
\begin{eqnarray*}
0\neq {\rm det}\left(
\begin{array}{cc}
{\bf i}\mu I_{n}  & -I_{n} \\
a_{0}^{0}-\lambda I_{n} & {\bf i}\mu I_{n} +a_{1}^{0}
\end{array}
\right)
={\rm det}\left(a_{0}^{0}-\mu^{2}I_{n}+{\bf i}a_{1}^{0}\mu-\lambda I_{n}\right)
\end{eqnarray*}
for all $\mu\in\mathbb{R}$.
This finishes the proof.
\end{proof}

By  Lemmas \ref{lm-admis} and \ref{lm-ED-constant-1},
we see that for each $\lambda\not\in \sigma(a_{0}^{0}-\mu^{2}I_{n}+{\bf i}a_{1}^{0}\mu)$ and each $F\in L^{2}(\mathbb{R})$,
there exists a unique solution $X$ in $L^{2}(\mathbb{R}^{2})$ satisfying
$\partial_{\xi}X=A_{1}(\lambda)X+F(\xi)$ and $\|X\|_{2}\leq K\|F\|_{2}$ for some constant $K>0$ independent of $F$.
This implies that for a given $f\in L^{2}(\mathbb{R})$,
equation $(\mathcal{L}^{1}_{0}-\lambda I)p=f$ has a unique solution $p$ in $H^{2}(\mathbb{R})$,
and $(\mathcal{L}^{1}_{0}-\lambda I)^{-1}$ is a bounded operator in $L^{2}(\mathbb{R})$. 
So $\lambda$ is in the resolvent set of $\mathcal{L}^{1}_{0}$.
In fact, we can further obtain that
\begin{eqnarray} \label{df-Sigma-01}
\Sigma_{0}^{1}=\bigcup_{\mu\in\mathbb{R}}\Sigma_{0}^{1}(\mu):=\bigcup_{\mu\in\mathbb{R}} \left\{\lambda\in\mathbb{C}:
{\rm det}\left(\lambda I_{n}-(a_{0}^{0}-\mu^{2}I_{n}+{\bf i}a_{1}^{0}\mu)\right)=0\right\}
\end{eqnarray}
is the spectrum of $\mathcal{L}^{1}_{0}$.

\subsection{The limiting operator in the second case}

Suppose that $D={\rm diag}(I_{m}, O_{n-m})$ ($0<m<n$) in \eqref{df-L}.
The operator $\mathcal{L}$ in this case is denoted by $\mathcal{L}^{2}$.
We split $p$ into two parts $p=(u,w)^{t}\in\mathbb{C}^{m}\times \mathbb{C}^{n-m}$,
and write $a_{1}(\xi,\epsilon)$ and $a_{0}(\xi,\epsilon)$ into
\begin{equation*}
a_{1}(\xi,\epsilon)=\left(
\begin{array}{cc}
a_{11}(\xi,\epsilon) &  a_{12}(\xi,\epsilon)\\
a_{21}(\xi,\epsilon) &  a_{22}(\xi,\epsilon)
\end{array}
\right), \ \ \ \
a_{0}(\xi,\epsilon)=\left(
\begin{array}{cc}
b_{11}(\xi,\epsilon) &  b_{12}(\xi,\epsilon)\\
b_{21}(\xi,\epsilon) &  b_{22}(\xi,\epsilon)
\end{array}
\right),
\end{equation*}
where $a_{11}(\xi,\epsilon)$ and $b_{11}(\xi,\epsilon)$ are $m\times m$ matrices,
and $a_{22}(\xi,\epsilon)$ and $b_{22}(\xi,\epsilon)$ are $(n-m)\times (n-m)$ matrices.
By {\bf (H2)}, the following limits exist:
\[
a_{ij}(\xi,\epsilon)\to a_{ij}^{0},\ \ \ b_{ij}(\xi,\epsilon)\to b_{ij}^{0},\ \ \ \mbox{ as }\ \ \epsilon\to 0.
\]
Here $a_{ij}^{0}$ and $b_{ij}^{0}$ are constant matrices.
Besides the hypotheses {\bf (H1)} and {\bf (H2)},
we further assume an additional hypothesis:
\begin{enumerate}
\item[{\bf (H3)}]
The matrices $a_{ij}$ satisfy
\begin{eqnarray*}
a_{12}(\xi,\epsilon)=0, \  \ a_{22}(\xi,\epsilon)=a_{22}(\epsilon), \ \ \ (\xi,\epsilon)\in \mathbb{R}\times (0,\epsilon_{0}),
\end{eqnarray*}
and $a_{22}^{0}$ and $a_{22}(\epsilon)$ for $\epsilon\in(0,\epsilon_{0})$ are invertible.
\end{enumerate}

The restrictions on $a_{22}$ and $a_{22}^{0}$ are natural for
the linearization of many nonlinear partial differential equations with well-posed conditions,
including the coupled model \eqref{PDE-ODE} as a typical model.
If the coefficient $a_{22}(\epsilon)$ of $\partial_{\xi} w$ is invertible,
then we can normalize the corresponding partial differential equation such that $a_{12}(\xi,\epsilon)=0$.

Let $Y=(u,\partial_{\xi} u,w)^{t}$.
Consider the eigenvalue problem $\mathcal{L}^{2}p=\lambda p$ on the whole real line.
It is equivalent to a first-order linear system of the form
\begin{eqnarray} \label{eq-period-2}
\ \ \ \ \ \partial_{\xi}Y=A_{2}(\xi,\lambda,\epsilon)Y,
\ \ \
A_{2}(\xi,\lambda,\epsilon)=\left(
\begin{array}{ccc}
0 & I_{m} & 0 \\
\lambda I_{m}-b_{11} & -a_{11} & -b_{12}\\
-a_{22}^{-1} b_{21} & -a_{22}^{-1} a_{21} &
a_{22}^{-1}\left(\lambda I_{n-m}-b_{22}\right)
\end{array}
\right),
\end{eqnarray}
where $A_{2}(\xi,\lambda,\epsilon)$ is a $(m+n)\times (m+n)$ matrix function.
For convenience, we drop $(\xi,\epsilon)$  dependence in $A_{2}$.
Letting $\epsilon\to 0$ in system \eqref{eq-period-2} yields
\begin{eqnarray} \label{eq-period-2-0}
\ \ \ \ \ \partial_{\xi}Y=A_{2}(\lambda)Y,
\ \ \
A_{2}(\lambda)=\left(
\begin{array}{ccc}
0 & I_{m} & 0 \\
\lambda I_{m}-b_{11}^{0} & -a_{11}^{0} & -b_{12}^{0}\\
-(a_{22}^{0})^{-1} b_{21}^{0} & -(a_{22}^{0})^{-1} a_{21}^{0} &
(a_{22}^{0})^{-1}\left(\lambda I_{n-m}-b_{22}^{0}\right)
\end{array}
\right).
\end{eqnarray}

Similarly to the first case (see Section \ref{sec-case-1}),
under the hypotheses {\bf (H2)} and {\bf (H3)},
we can obtain that the limiting operator $\mathcal{L}_{0}^{2}$ of $\mathcal{L}^{2}$  is in the form
\begin{eqnarray} \label{df-L-2}
\mathcal{L}_{0}^{2}
\left(
\begin{array}{c}
u \\
w
\end{array}
\right)
:= \left(
\begin{array}{c}
\partial_{\xi}^{2}u+a_{11}^{0}\partial_{\xi} u+b_{11}^{0}u+b_{12}^{0}w \\
a_{21}^{0}\partial_{\xi} u+a_{22}^{0}\partial_{\xi} w+b_{21}^{0}u+b_{22}^{0}w
\end{array}
\right),
\end{eqnarray}
where $\mathcal{L}_{0}^{2}: H^{2}(\mathbb{R})\times H^{1}(\mathbb{R})\subset L^{2}(\mathbb{R})\times L^{2}(\mathbb{R})
\to L^{2}(\mathbb{R})\times L^{2}(\mathbb{R})$.
Similarly to Lemma \ref{lm-ED-constant-1}, we have the following result.
\begin{lemma} \label{lm-ED-constant-2}
Suppose that the hypothesis {\bf (H3)} holds.
Then for a given $\lambda\in \mathbb{C}$, system \eqref{eq-period-2-0} admits an exponential dichotomy
if and only if
\begin{eqnarray}\label{ED-cond-2}
\ \ \ \  \Delta_{2}(\lambda,\mu):=
{\rm det}\left(
\begin{array}{cc}
(\lambda+\mu^{2}) I_{m}-{\bf i}\mu a_{11}^{0}-b_{11}^{0}  & -b_{12}^{0} \\
-(b_{21}^{0}+{\bf i}\mu a_{21}^{0}) & \lambda I_{n-m}-{\bf i}\mu a_{22}^{0}-b_{22}^{0}
\end{array}
\right)
\neq 0,\ \ \ \mu\in\mathbb{R}.
\end{eqnarray}
\end{lemma}
\begin{proof}
By a direct computation, we have
\begin{eqnarray*}
&&\lefteqn{{\rm det}({\bf i}\mu I_{m+n}-A_{2}(\lambda))}\\
&&=
{\rm det}\left(
\begin{array}{ccc}
{\bf i}\mu I_{m} & -I_{m} & 0 \\
-\lambda I_{m}+b_{11}^{0} & {\bf i}\mu I_{m}+a_{11}^{0} & b_{12}^{0}\\
(a_{22}^{0})^{-1} b_{21}^{0} & (a_{22}^{0})^{-1} a_{21}^{0} &
{\bf i}\mu I_{n-m}-(a_{22}^{0})^{-1}\left(\lambda I_{n-m}-b_{22}^{0}\right)
\end{array}
\right)\\
&&=
{\rm det}\left(
\begin{array}{ccc}
I_{m} & 0 &  0 \\
-{\bf i}\mu I_{m}-a_{11}^{0} & {\bf i}\mu a_{11}^{0}+b_{11}^{0}-(\lambda+\mu^{2}) I_{m} &  b_{12}^{0}\\
-(a_{22}^{0})^{-1} a_{21}^{0} & (a_{22}^{0})^{-1} b_{21}^{0}+{\bf i}\mu(a_{22}^{0})^{-1} a_{21}^{0} &
{\bf i}\mu I_{n-m}-(a_{22}^{0})^{-1}\left(\lambda I_{n-m}-b_{22}^{0}\right)
\end{array}
\right)\\
&&=
(-1)^{n}({\rm det}(a_{22}^{0}))^{-1}\Delta_{2}(\lambda,\mu).
\end{eqnarray*}
Therefore, the last equality together with {\bf (H3)} yields this lemma.
\end{proof}

Similarly to the definition of the spectrum $\Sigma_{0}^{1}$ for $\mathcal{L}^{1}_{0}$,
and by  Lemmas \ref{lm-admis} and \ref{lm-ED-constant-2},
we define the spectrum $\Sigma_{0}^{2}$ of $\mathcal{L}^{2}_{0}$ by
\begin{eqnarray}\label{df-Sigma-02}
\Sigma_{0}^{2}=\bigcup_{\mu\in\mathbb{R}}\Sigma_{0}^{2}(\mu)
    :=\bigcup_{\mu\in\mathbb{R}} \left\{\lambda\in\mathbb{C}: \Delta_{2}(\lambda,\mu)=0\right\}.
\end{eqnarray}

\subsection{The limiting operator in the third case}

Suppose that $D=O_{n}$ in \eqref{df-L} and $a_{1}(\xi,\epsilon)=a_{1}(\epsilon)$.
Then the operator $\mathcal{L}$ is reduced to $\mathcal{L}^{3}$ of the form
\begin{eqnarray} \label{df-L-3}
\mathcal{L}^{3}p=a_{1}(\epsilon)\partial_{\xi}p+a_{0}(\xi,\epsilon)p,
\end{eqnarray}
where
$$\mathcal{L}^{3}: H^{1}(\mathbb{R})\subset L^{2}(\mathbb{R}) \to L^{2}(\mathbb{R}).$$
Besides the hypotheses {\bf (H1)} and {\bf (H2)},
we further assume an additional hypothesis:
\begin{enumerate}
\item[{\bf (H4)}]
The matrices $a_{1}^{0}$ and $a_{1}(\epsilon)$ for $\epsilon\in(0,\epsilon_{0})$ are invertible .
\end{enumerate}

The hypothesis {\bf (H4)} is  natural for many evolutionary systems with well-posed conditions.
We write the eigenvalue problem $\mathcal{L}^{3}p=\lambda p$ as
\begin{eqnarray*}\label{eq-period-3}
\partial_{\xi}p=A_{3}(\xi,\lambda,\epsilon)p,
\ \ \ \ \
A_{3}(\xi,\lambda,\epsilon)=(a_{1}(\epsilon))^{-1}(\lambda I_{n}-a_{0}(\xi,\epsilon)).
\end{eqnarray*}
The limiting system and the limiting operator $\mathcal{L}_{3}^{0}$ are
\begin{eqnarray}  \label{eq-period-3-0}
\partial_{\xi}p=A_{3}(\lambda)p,
\ \ \ \ \
A_{3}(\lambda)=(a_{1}^{0})^{-1}(\lambda I_{n}-a_{0}^{0}),
\end{eqnarray}
and
\begin{eqnarray*}
\mathcal{L}^{3}_{0}p=a_{1}^{0}\partial_{\xi}p+a_{0}^{0}p,
\ \ \ \ \ p\in H^{1}(\mathbb{R}),
\end{eqnarray*}
respectively.
Similarly to Lemmas \ref{lm-ED-constant-1} and \ref{lm-ED-constant-2},
we have the following result.
\begin{lemma} \label{lm-ED-constant-3}
Suppose that the hypothesis {\bf (H4)} holds. Then
for a given $\lambda\in \mathbb{C}$, system \eqref{eq-period-3-0} admits an exponential dichotomy
if and only if
\begin{eqnarray}\label{ED-cond-3}
\Delta_{3}(\lambda,\mu):=
{\rm det}\left(
\lambda I_{n}-(a_{0}^{0}+{\bf i}\mu a_{1}^{0})
\right)
\neq 0, \ \ \ \ \mu\in\mathbb{R}.
\end{eqnarray}
\end{lemma}
Finally, we define the spectrum $\Sigma_{0}^{3}$ of $\mathcal{L}^{3}_{0}$ by
\begin{eqnarray} \label{df-Sigma-03}
\Sigma_{0}^{3}=\bigcup_{\mu\in\mathbb{R}}\Sigma_{0}^{3}(\mu)
    :=\bigcup_{\mu\in\mathbb{R}} \left\{\lambda\in\mathbb{C}: \Delta_{3}(\lambda,\mu)=0\right\}.
\end{eqnarray}

\subsection{The spectra for the limiting operators}

Although the spectra of  $\mathcal{L}^{i}_{0}$ ($i=1,2,3$) contain purely essential spectra,
we define the {\it generalized algebraic multiplicity} $m_{ga}(\lambda_{0})$ of
the spectral point $\lambda_{0}\in \Sigma_{0}^{i}$ ($i=1,2,3$) in the following way.
For each $\lambda_{0}\in \Sigma_{0}^{i}$,
$$
\Delta_{i}(\lambda_{0},\mu)=0
$$
is an algebraic equation in $\mu$.
Let $\mu_{1}(\lambda_{0}), ..., \mu_{m_{0}}(\lambda_{0})$ denote all real roots of this equation.
Then the {\it generalized algebraic multiplicity} $m_{ga}(\lambda_{0})$ of $\lambda_{0}$ is given by
\begin{eqnarray}\label{df-gam}
m_{ga}(\lambda_{0})=\sum_{j=1}^{m_{0}} m_{a}(\lambda_{0},\mu_{j}(\lambda_{0})),
\end{eqnarray}
where $m_{a}(\lambda_{0},\mu_{j}(\lambda_{0}))$ are the orders of the zero $\lambda_{0}$
for the equations
\begin{eqnarray*}
\Delta_{i}(\lambda,\mu_{j}(\lambda_{0}))=0, \ \ \ i=1,2,3, \ \ j=1,...,m_{0}.
\end{eqnarray*}
Next we summarize the properties for the spectra $\Sigma_{0}^{i}$.
\begin{lemma} \label{lm-zero-spect}
The spectra $\Sigma_{0}^{i}$ ($i=1,2,3$) has  the following properties:
\begin{enumerate}
\item[(i)]
There are $n$ continuous functions $\lambda_{j}(\mu)$ for $j=1,2,..,n$ and $\mu\in\mathbb{R}$
such that
\begin{eqnarray*}
\Sigma_{0}^{i}:=\left\{\lambda\in\mathbb{C}:
\lambda=\lambda_{j}(\mu),\ j=1,2,...,n,\ \mu\in\mathbb{R} \right\}.
\end{eqnarray*}

\item[(ii)]
If the matrix $a_{0}^{0}$ has eigenvalues with positive real part,
then
\[
\Sigma_{0}^{i}\bigcap \left\{{\lambda\in\mathbb{C}:{\rm Re}\lambda>0}\right\}
\neq \varnothing.
\]
\end{enumerate}
\end{lemma}
\begin{proof}
We write each $\Delta_{i}(\lambda,\mu)$ ($i=1,2,3$) as
\begin{eqnarray*}
\Delta_{i}(\lambda,\mu)=\lambda^{n}+\kappa_{n-1}(\mu)\lambda^{n-1}+\cdot\cdot\cdot+\kappa_{1}(\mu)\lambda^{1}+\kappa_{0}(\mu),
\end{eqnarray*}
where $\kappa_{j}$ are continuous in $\mu$ and determined by the elements of the matrices $a_{0}^{0}$ and $a_{1}^{0}$.
Note that for each $\mu\in\mathbb{R}$,
the equation $\Delta_{i}(\lambda,\mu)=0$ has $n$ complex roots $\lambda_{j}(\mu)$ for $j=1,2,...,n$.
By the  Rouch\'e Theorem \cite[p.247]{Dieudonne-69} and continuity of the functions $\kappa_{j}$,
we can further conclude that $\lambda_{j}(\mu)$  are continuous in $\mu$.
This finishes the proof for (i).

Letting $\mu=0$ in $\Delta_{i}(\lambda,\mu)=0$ yields that
\begin{eqnarray*}
\Delta_{i}(\lambda,0)={\rm det}(\lambda I_{n}-a_{0}^{0}).
\end{eqnarray*}
This implies that $\lambda_{j}(0)$, $j=1,2,...,n$, are all eigenvalues of the matrix $a_{0}^{0}$.
Since the matrix $a_{0}^{0}$ has eigenvalues with positive real parts,
there exist some $\lambda_{j}(0)$ with ${\rm Re}\lambda_{j}(0)>0$.
We finish the proof for  (ii) by continuity of $\lambda_{j}$.
Therefore, the proof is now complete.
\end{proof}

\section{Perturbation theory for the spectra}
\label{sec-pert}

In this section, we consider the perturbation of the spectrum for the linear operator $\mathcal{L}$ (see \eqref{df-L}).
Recall that the spectral problem $\mathcal{L} p=\lambda p$ is equivalent to the following system
\begin{eqnarray} \label{eq-spect}
\partial_{\xi}Y=A(\xi,\lambda,\epsilon)Y,
\end{eqnarray}
where
$A(\xi,\lambda,\epsilon)=A_{i}(\xi,\lambda,\epsilon)$ for $i=1,2,3,4$,
and $A_{i}(\xi,\lambda,\epsilon)$ ($i=1,2,3$) are defined as in
\eqref{eq-period-1}, \eqref{eq-period-2} and \eqref{eq-period-3}, respectively.

Note that $A(\xi,\lambda,\epsilon)=A(\xi+T_{\epsilon},\lambda,\epsilon)$.
By Lemma \ref{lm-admis}, the spectrum $\sigma(\mathcal{L})$ of $\mathcal{L}$ has purely essential spectrum.
See \cite{Kapitula-Promislow-13,Sandstede-02} for the detailed proof.
Furthermore, $\lambda$ is in $\sigma(\mathcal{L})$ if and only if
there exists a  $\mu$ in $(-\pi/T_{\epsilon},\pi/T_{\epsilon}]$ such that
\eqref{eq-spect} has a nonzero solution $Y$ satisfying
\begin{eqnarray}\label{BVD}
Y(T_{\epsilon})=e^{{\bf i}\mu T_{\epsilon}}Y(0).
\end{eqnarray}
We refer the detailed discussion in \cite[pp. 68-69]{Kapitula-Promislow-13}.
By a change of variables
\[
Y(\xi)\to Y(\xi)e^{{\bf i}\mu \xi},
\]
we transform system \eqref{eq-spect} with the boundary value condition \eqref{BVD} into
\begin{eqnarray} \label{eq-evans-2}
\partial_{\xi}Y=(A(\xi,\lambda,\epsilon)-{\bf i}\mu I)Y,
\end{eqnarray}
which is subject to the boundary value condition
$$Y(T_{\epsilon})=Y(0).$$
Throughout this section, we always use $I$ to denote the identity matrix $I_{2n}$, $I_{m+n}$ or $I_{n}$.

Let $\Psi(\xi,\lambda,\mu,\epsilon)$ be the principal fundamental matrix solution of the periodic system \eqref{eq-evans-2},
and $E(\lambda,\mu,\epsilon)$ be defined by
\begin{eqnarray*}
E(\lambda,\mu,\epsilon)={\rm det}(\Psi(T_{\epsilon},\lambda,\mu,\epsilon)-I)=0, \ \ \
 (\lambda,\mu,\epsilon)\in \mathbb{C}\times (-\pi/T_{\epsilon},\pi/T_{\epsilon}]\times (0,\epsilon_{0}).
\end{eqnarray*}
The function $E(\lambda,\mu,\epsilon)$ for each $\epsilon\in(0,\epsilon_{0})$
is called the {\it Evans function} of the eigenvalue problem $\mathcal{L}p=\lambda p$
(see \cite{Gardner-93,Kapitula-Promislow-13,Sandstede-02}).
By Lemmas 8.4.1 and 8.4.2 in \cite[p.242]{Kapitula-Promislow-13}, we have the next lemma.
\begin{lemma}\label{lm-spect-0}
For a given $\epsilon\in(0,\epsilon_{0})$,
the spectral parameter $\lambda$ is in $\sigma (\mathcal{L})$ if and only if
there exists a constant $\mu \in\mathbb{R}$ such that
$E(\lambda,\mu,\epsilon)=0$.
\end{lemma}

In order to consider the zeros of $E(\cdot,\cdot,\epsilon)$ for  $\epsilon\in(0,\epsilon_{0})$,
we  begin by studying an auxiliary perturbed system
\begin{eqnarray} \label{eq-period-1-pt}
\partial_{\xi}Y=\left(A(\lambda)-{\bf i}\mu I+sB(\xi,\epsilon)\right)Y,\ \ \ \ \xi\in\mathbb{R},
\end{eqnarray}
where  $s\in[0,1]$ and $\lambda\in\mathbb{C}$.
According to the form of $D$ in \eqref{df-L},
the matrix $A(\lambda)$ is defined by
\begin{eqnarray*}
A(\lambda)=\left\{
\begin{aligned}
&A_{1}(\lambda), && \mbox{ if }\ D=I_{n},\\
&A_{2}(\lambda), && \mbox{ if }\ D={\rm diag}(I_{m},O_{n-m}),\\
&A_{3}(\lambda), && \mbox{ if }\ D=O_{n},
\end{aligned}
\right.
\end{eqnarray*}
and the corresponding $B(\xi,\epsilon)$ is given by
\begin{eqnarray*}
B(\xi,\epsilon)=A_{i}(\xi,\lambda,\epsilon)-A_{i}(\lambda), \ \ \ \ i=1,2,3.
\end{eqnarray*}

When $s=0$,
system \eqref{eq-period-1-pt} is reduced to
\begin{eqnarray} \label{eq-evans-1}
\partial_{\xi}Y=\left(A(\lambda)-{\bf i}\mu I\right)Y.
\end{eqnarray}
Define $\mathcal{B}: (0,\epsilon_{0})\to \mathbb{R}^{+}$  by
\[
\mathcal{B}(\epsilon)=\max_{\xi\in[0,T_{\epsilon}]}|B(\xi,\epsilon)|.
\]
By {\bf (H2)}, we have $\mathcal{B}(\epsilon)=O(\epsilon)$ for $\epsilon\in (0,\epsilon_{0})$.
We see that  \eqref{eq-period-1-pt} is a perturbed system of \eqref{eq-evans-1}.
By  the preceding discussion,
this unperturbed system  \eqref{eq-period-1-pt} loses exponential dichotomies at the zeros of the following function:
\begin{eqnarray*}
\Delta(\lambda,\mu)=\left\{
\begin{aligned}
&\Delta_{1}(\lambda,\mu), && \mbox{ if }\ D=I_{n},\\
&\Delta_{2}(\lambda,\mu), && \mbox{ if }\ D={\rm diag}(I_{m},O_{n-m}),\\
&\Delta_{3}(\lambda,\mu), && \mbox{ if }\ D=O_{n},
\end{aligned}
\right.
\end{eqnarray*}
where $\Delta_{1}(\lambda,\mu)$, $\Delta_{2}(\lambda,\mu)$ and $\Delta_{3}(\lambda,\mu)$
are defined by \eqref{ED-cond}, \eqref{ED-cond-2} and \eqref{ED-cond-3}, respectively.

Similarly to $E$, we define $\mathcal{E}$ by
\begin{eqnarray}\label{df-F-mu}
\mathcal{E}(s,\lambda,\mu,\epsilon)=
{\rm det}\left(\Phi(T_{\epsilon},s,\lambda,\mu,\epsilon)- I\right),\ \ \
(s,\lambda,\mu,\epsilon)\in [0,1]\times \mathbb{C}\times \mathbb{R}\times (0,\epsilon_{0}),
\end{eqnarray}
where $\Phi(\cdot,s,\lambda,\mu,\epsilon)$ is
the principal fundamental matrix solution of \eqref{eq-period-1-pt}
for each $(s,\lambda,\mu,\epsilon)\in [0,1]\times \mathbb{C}\times \mathbb{R}\times (0,\epsilon_{0})$.

\begin{lemma}\label{lm-spect}
The function $\mathcal{E}$ has the following properties:
\begin{enumerate}
\item[(i)]
$\mathcal{E}(1,\lambda,\mu,\epsilon)=E(\lambda,\mu,\epsilon)$.

\item[(ii)]
For each $(s,\lambda,\epsilon)\in [0,1]\times \mathbb{C}\times (0,\epsilon_{0})$,
the function $\mathcal{E}(s,\lambda,\cdot,\epsilon)$ has period $2\pi/T_{\epsilon}$,
where $T_{\epsilon}$ is defined as in {\bf (H1)}.

\item[(iii)] $\mathcal{E}(0,\lambda_{j}(\mu),\mu,\epsilon)=0$ for each $\mu\in\mathbb{R}$
and $j=1,2,...,n$, where $\lambda_{j}(\mu)$ are defined as
in Lemma \ref{lm-zero-spect}.
\end{enumerate}
\end{lemma}
\begin{proof}
Letting $s=1$ in \eqref{eq-period-1-pt}, we have (i).
By a direct computation,
\begin{eqnarray*}
\mathcal{E}(s,\lambda,\mu,\epsilon)=
{\rm det}\left(\Phi(T_{\epsilon},s,\lambda,0,\epsilon)\exp({\bf i}\mu T_{\epsilon})-I_{2n}\right).
\end{eqnarray*}
This yields (ii).
Finally, we prove (iii). Note that  $\Delta(\lambda_{j}(\mu),\mu)=0$.
Then $(A(\lambda_{j}(\mu))-{\bf i}\mu I_{n})$ has zero as an eigenvalue,
which  yields (iii). Then the proof is finished.
\end{proof}

We remark that the graphs of  $\lambda_{j}(\cdot)$ are called the spectral curves
of the limiting operators $\mathcal{L}_{0}^{i}$ ($i=1,2,3$).
Now we prove a lemma which plays an important role in the subsequent proof.

\begin{lemma}\label{lm-zero-1}
Suppose that $\mu_{0}$ in $\mathbb{R}$ satisfies $\Sigma^{i}_{0}(\mu_{0})\neq \varnothing$ for $i=1,2,3$,
where $\Sigma^{i}_{0}(\mu_{0})$ for $i=1,2,3$ are defined in \eqref{df-Sigma-01}, \eqref{df-Sigma-02} and  \eqref{df-Sigma-03},
respectively.
Let $U(\lambda_{j}(\mu_{0}))$ be a bounded open neighborhood of  $\lambda_{j}(\mu_{0})$ in $\mathbb{C}$
such that $\lambda_{j}(\mu_{0})$ is the only point of $\Sigma^{i}_{0}(\mu_{0})$  in $U(\lambda_{j}(\mu_{0}))$,
and all eigenvalues $\alpha_{j}(\lambda)+{\bf i}\beta_{j}(\lambda)$ of  $A(\lambda)$ satisfy
\begin{eqnarray}\label{condt-U-j}
0<|\alpha_{j}(\lambda)|+|\beta_{j}(\lambda)-\mu_{0}|<\frac{2\pi}{T_{\epsilon}},
\ \ \ \ \ \lambda\in \partial U(\lambda_{j}(\mu_{0})),  \ \ \epsilon\in (0,\epsilon_{0}).
\end{eqnarray}
Then there exists a small $\tilde{\epsilon}_{0}$ with $0<\tilde{\epsilon}_{0}\leq \epsilon_{0}$
such that for each $s\in [0,1]$ and $\epsilon\in (0,\tilde{\epsilon}_{0})$,
the number of the zeros for $\mathcal{E}(s,\cdot,\mu_{0},\epsilon)$ in $U(\lambda_{j}(\mu_{0}))$
is equal to the order of the zero $\lambda=\lambda_{j}(\mu_{0})$  for $\Delta(\lambda,\mu_{0})$.
\end{lemma}
\begin{proof}
For a given $\mu\in\mathbb{R}$,
the function $\Delta(\cdot,\mu)$ is analytic, which implies that  all zeros of $\Delta(\cdot,\mu)$ are isolated.
Then we can choose  $U(\lambda_{j}(\mu_{0}))$  as small as possible such that
$U(\lambda_{j}(\mu_{0}))$ fulfils the properties in this lemma.

Now we prove this lemma by the Brouwer degree theory.
By the continuous dependency on parameters,
we obtain that
$\{\mathcal{E}(s,\cdot,\mu_{0},\epsilon): 0\leq s\leq 1\}$ is  a family of  continuous homotopies.
The remaining  proof is divided into three steps.

{\bf Step 1.} We prove  that
$\mathcal{E}(0,\cdot,\mu_{0},\epsilon)$ has a zero in $U(\lambda_{j}(\mu_{0}))$ and satisfies
$\mathcal{E}(0,\lambda,\mu_{0},\epsilon)\neq 0$ for each $\lambda\in \partial U(\lambda_{j}(\mu_{0}))$.
By \eqref{eq-evans-1}, we have
\begin{eqnarray*}
\mathcal{E}(0,\lambda,\mu_{0},\epsilon)=
{\rm det}\left(\exp\left((A(\lambda)-{\bf i}\mu_{0} I)T_{\epsilon}\right)-I\right).
\end{eqnarray*}
Since $\Delta(\lambda_{j}(\mu_{0}),\mu_{0})=0$,
we have that $\mathcal{E}(0,\lambda,\mu_{0},\epsilon)$ has a zero $\lambda=\lambda_{j}(\mu_{0})$ in  $U(\lambda_{j}(\mu_{0}))$.
Note that $\mathcal{E}(0,\lambda,\mu_{0},\epsilon)=0$ if and only if
$A(\lambda)$ has an eigenvalue of the form ${\bf i}(\mu_{0}+2k\pi/T_{\epsilon})$ for $k\in \mathbb{Z}$.
This together with  \eqref{condt-U-j} yields
that $\mathcal{E}(0,\lambda,\mu_{0},\epsilon)\neq 0$ for each $\lambda\in \partial U(\lambda_{j}(\mu_{0}))$.
By Lemma \ref{lm-BDT-ANA} and \eqref{df-BDT},
\begin{eqnarray}\label{zero-cond}
\begin{aligned}
d_{B}(\mathcal{E}(0,\cdot,\mu_{0},\epsilon),U(\lambda_{j}(\mu_{0})),0)
=&\ \mbox{the order of the zero $\lambda=\lambda_{j}(\mu_{0})$  for $\mathcal{E}(0,\cdot,\mu_{0},\epsilon)$}\\
=&\ \mbox{the order of the zero $\lambda=\lambda_{j}(\mu_{0})$  for $\Delta(\lambda,\mu_{0})$},
\end{aligned}
\end{eqnarray}
where $d_{B}$ is defined by \eqref{df-BDT}.

{\bf Step 2.}  We prove that
there exists a small $\tilde{\epsilon}_{0}$ with $0<\tilde{\epsilon}_{0}\leq \epsilon_{0}$
such that
\begin{eqnarray}\label{cond-F}
\mathcal{E}(s,\lambda,\mu_{0},\epsilon)\neq 0, \ \ \
(s,\lambda,\epsilon)\in [0,1]\times\partial U(\lambda_{j}(\mu_{0}))\times (0,\tilde{\epsilon}_{0}).
\end{eqnarray}

Let $\alpha_{j}(\lambda,s,\epsilon)+{\bf i}\beta_{j}(\lambda,s,\epsilon)$, $j=1,2,...,n$, denote
all  Floquet exponents of system \eqref{eq-period-1-pt} with $\mu=\mu_{0}$.
Note that for sufficiently small $|\epsilon|>0$,
$$|sB(\xi,\epsilon)|\leq \mathcal{B}(\epsilon)=O(\epsilon)$$
for each $\xi\in\mathbb{R}$ and each $s\in[0,1]$.
Then  by \eqref{condt-U-j} and (ii) in Lemma \ref{lm-Period},
there exists a small $\tilde{\epsilon}_{0}$ with $0<\tilde{\epsilon}_{0}\leq \epsilon_{0}$
such that
\begin{eqnarray*}
|\alpha_{j}(\lambda,s,\epsilon)|+|\beta_{j}(\lambda,s,\epsilon)|\neq 0, \ \ 0<|\alpha_{j}(\lambda,s,\epsilon)|+|\beta_{j}(\lambda,s,\epsilon)-\mu_{0}|< 2\pi/T_{\epsilon}
\end{eqnarray*}
for each $(s,\lambda,\epsilon)\in [0,1]\times\partial U(\lambda_{j}(\mu_{0}))\times (0,\tilde{\epsilon}_{0})$.
This implies \eqref{cond-F}.

{\bf Step 3.} We finish the proof by (iii) in Lemma \ref{lm-BDT-1}.
After checking the conditions in (iii) in Lemma \ref{lm-BDT-1},
we obtain that $d_{B}(\mathcal{E}(s,\cdot,\mu_{0},\epsilon),U(\lambda_{j}(\mu_{0})),0)$ is a constant
for every $s\in [0,1]$ and $\epsilon\in (0,\tilde{\epsilon}_{0})$,
that is,
\begin{eqnarray*}
d_{B}(\mathcal{E}(0,\cdot,\mu_{0},\epsilon),U(\lambda_{j}(\mu_{0})),0)
=d_{B}(\mathcal{E}(s,\cdot,\mu_{0},\epsilon),U(\lambda_{j}(\mu_{0})),0)
\ \mbox{ for }\ (s,\epsilon)\in[0,1]\times (0,\tilde{\epsilon}_{0}).
\end{eqnarray*}
Therefore, the proof is finished by \eqref{df-BDT}, \eqref{zero-cond} and Lemma \ref{lm-BDT-ANA}.
\end{proof}

\begin{remark}
By the Bloch-wave decomposition,
i.e., taking a change $p\to e^{{\bf i}\mu \xi}p$ for $\mu\in(-\pi/T_{\epsilon},\pi/T_{\epsilon}]$,
the operator $\mathcal{L}$ is changed into
\begin{eqnarray*}
\mathcal{L}_{\mu}p:=D(\partial_{\xi}+{\bf i}\mu I)^{2}p+a_{1}(\xi,\epsilon)(\partial_{\xi}+{\bf i}\mu I)p+a_{0}(\xi,\epsilon)p,
\ \ \ \mu\in(-\frac{\pi}{T_{\epsilon}},\frac{\pi}{T_{\epsilon}}],
\end{eqnarray*}
where $\mathcal{L}_{\mu}: L^{2}_{\rm per}([0,T_{\epsilon}])\to L^{2}_{\rm per}([0,T_{\epsilon}])$.
The domains of $\mathcal{L}_{\mu}$ in the first, second and third cases
are the spaces $H^{2}_{\rm per}([0,T_{\epsilon}])$,
$H^{2}_{\rm per}([0,T_{\epsilon}])\times H^{1}_{\rm per}([0,T_{\epsilon}])$
and $H^{1}_{\rm per}([0,T_{\epsilon}])$, respectively.
Then we have
\begin{eqnarray*}
\sigma(\mathcal{L})=\bigcup_{-\pi/T_{\epsilon}<\mu\leq \pi/T_{\epsilon}}\sigma_{\rm pt}(\mathcal{L}_{\mu}).
\end{eqnarray*}
See, for instance, \cite[Section 3]{Chen-Duan-21} or \cite[p.240]{Kapitula-Promislow-13}.
Let $\tilde{\lambda}_{1},...,\tilde{\lambda}_{m}$ denote all eigenvalues of ${L}_{\mu_{0}}$ in $U(\lambda_{j}(\mu_{0}))$
and $\tilde{\kappa}_{1},...,\tilde{\kappa}_{m}$ denote their algebraic multiplicities,
where $\mu_{0}$ and $U(\lambda_{j}(\mu_{0}))$ are defined as in Lemma \ref{lm-zero-1}.
By  \cite[Lemma 8.4.1]{Kapitula-Promislow-13},
we see that the algebraic multiplicity of the eigenvalue $\lambda$ of ${L}_{\mu}$ is equal to
the multiplicity of the zero $\lambda$ of $E(\lambda,\mu,\epsilon)$.
This together with Lemma \ref{lm-zero-1} yields that
$(\tilde{\kappa}_{1}+\cdot\cdot\cdot+\tilde{\kappa}_{m})$ is equal to the order of the zero $\lambda=\lambda_{j}(\mu_{0})$  for $\Delta(\lambda,\mu_{0})$.
\end{remark}

By the above lemma,  we prove the main result in this section.
\begin{theorem}\label{thm-1}{\rm (Perturbation of the spectrum)}
Suppose that the linear operator $\mathcal{L}$ defined by  \eqref{df-L} 
satisfies {\bf (H1)} and {\bf (H2)} in the first case,
{\bf (H1)}-{\bf (H3)} in the second case,
or {\bf (H1)}, {\bf (H2)} and {\bf (H4)} in the third case.
Then for each $i=1,2,3$, the following statements hold:
\begin{enumerate}
\item[(i)]
Suppose that $\lambda=\lambda_{0}$ is in the spectrum of $\mathcal{L}_{0}^{i}$.
Then its  generalized algebraic multiplicity $m_{ga}(\lambda_{0})$ is finite.
Furthermore, there exists a small $\tilde{\epsilon}_{0}$ with $0<\tilde{\epsilon}_{0}\leq \epsilon_{0}$
such that $\mathcal{L}$ with $0<\epsilon<\tilde{\epsilon}_{0}$ has  $m_{ga}(\lambda_{0})$ spectral points
$\tilde{\lambda}_{1}(\epsilon)$,...,$\tilde{\lambda}_{m_{ga}(\lambda_{0})}(\epsilon)$,
which are perturbed from $\lambda_{0}$ and satisfy that
$\tilde{\lambda}_{j}(\epsilon)\to \lambda_{0}$ as $\epsilon\to 0$.

\item[(ii)] For a certain $j\in \{1,2,...,n\}$,
let $\mathcal{I}$ be a connected closed interval in the domain of $\lambda_{j}(\cdot)$
such that for each $\mu\in\mathcal{I}$,
\begin{eqnarray}\label{cond-kT}
\ \ \ \mu\neq k\pi/T_{\epsilon}  \mbox{ for }\ \epsilon\in (0,\epsilon_{0}) \mbox{ and } k\in\mathbb{Z},\ \
\lambda_{j}(\mu)\cap \lambda_{i}(\mu)=\varnothing  \mbox{ for } j\neq i \mbox{ and } \mu\in\mathcal{I}.
\end{eqnarray}
Then there exists a small $\tilde{\epsilon}_{0}$ with $0<\tilde{\epsilon}_{0}\leq \epsilon_{0}$
such that the spectrum of the linear operator $\mathcal{L}$ with $0<\epsilon<\tilde{\epsilon}_{0}$ has
a continuous curve $\lambda_{j,\epsilon}(\mu)$ for $\mu\in\mathcal{I}$,
which is perturbed from $\lambda_{j}(\mu)$ and satisfy that
$\lambda_{j,\epsilon}(\mu)\to \lambda_{j}(\mu)$ as $\epsilon\to 0$
in the sense of the Hausdorff distance.

\item[(iii)] If the matrix $a_{0}^{0}$ has eigenvalues with positive real part,
then for sufficiently small $\epsilon$, the operator $\mathcal{L}$ is spectrally unstable.
\end{enumerate}
\end{theorem}
\begin{proof}
We first prove (i).
Note that $\Delta(\lambda,\mu)$ is analytic in $\lambda$ and $\mu$.
Then by the definition of $m_{ga}(\lambda)$ in \eqref{df-gam},
we can prove the first statement in (i).
Without loss of generality,
we assume that $\Delta(\lambda_{0},\mu)=0$ has exactly one real zero $\mu=\mu_{0}$.
Then $m_{ga}(\lambda_{0})=m_{a}(\lambda_{0},\mu_{0})$.
By the same argument in the proof for Lemma \ref{lm-zero-1},
there exists an open ball $U(\lambda_{0},\delta)$ in $\mathbb{C}$
with the center at $\lambda_{0}$ and the radius $\delta$
such that $\{\lambda_{0}\}=U(\lambda_{0},\delta)\bigcap \Sigma^{i}_{0}(\mu_{0})$ ($i=1,2,3$),
and all eigenvalues $\alpha_{i}(\lambda)+{\bf i}\beta_{i}(\lambda)$ of  $A(\lambda)$ for $\lambda\in U(\lambda_{0},\delta)$
satisfy \eqref{condt-U-j}, where $U(\lambda_{0},\delta)$ can shrink to $\lambda_{0}$ as $\epsilon\to0$.
By Lemma \ref{lm-zero-1},
the function $\mathcal{E}(1,\cdot,\mu_{0},\epsilon)$ has $m_{ga}(\lambda_{0})$ zeros $\tilde{\lambda}_{1}(\epsilon)$,...,$\tilde{\lambda}_{m_{ga}(\lambda_{0})}(\epsilon)$, which are in the interior of  $U(\lambda_{0},\delta)$
and satisfies that $\tilde{\lambda}_{j}(\epsilon)\to \lambda_{0}$ as $\epsilon\to 0$.
Therefore,  (i) is proved  by Lemma \ref{lm-spect-0} and (i) in Lemma \ref{lm-spect}.

Secondly, we prove (ii).
By \eqref{cond-kT}, there exists a small $\delta>0$ such that
for each $(\lambda,\mu)\in B(\lambda_{j}(\mu),\delta)\times \mathcal{I}\subset \mathbb{C}\times \mathbb{R}$
and each $\epsilon\in (0,\epsilon_{0})$, we have $\mathcal{E}(0,\lambda,\mu,\epsilon)\neq 0$.
By continuity of $\mathcal{E}(s,\lambda,\mu,\epsilon)$ with respect to $(\lambda,\mu)$,
the compactness of $[0,1]\times B(\lambda_{j}(\mu),\delta)\times \mathcal{I}$
in $\mathbb{R}\times\mathbb{C}\times \mathbb{R}$ and Lemma \ref{lm-Period},
there exists a small $\tilde{\epsilon}_{0}$ with $0<\tilde{\epsilon}_{0}\leq \epsilon_{0}$ such that
$\mathcal{E}(s,\lambda,\mu,\epsilon)\neq 0$ for $(s,\lambda,\mu)\in[0,1]\times B(\lambda_{j}(\mu),\delta)\times \mathcal{I}$.
Without loss of generality, we assume that $m_{ga}(\lambda_{j}(\mu))=1$ for each $\mu\in\mathcal{I}$.
Then similarly to the proof for Lemma \ref{lm-zero-1},
there exists a function $\lambda_{j,\epsilon}(\mu)$ perturbed from $\lambda_{j}(\mu)$ for $\mu\in\mathcal{I}$.
Note that the zeros of $\mathcal{E}(s,\cdot,\mu,\epsilon)$ is continuous in $\mu$.
Then $\lambda_{j,\epsilon}(\mu)$ is continuous in $\mu$.
Thus, (ii) is proved.

Finally, we prove (iii).
If the matrix $a_{0}^{0}$ has eigenvalues with positive real part,
then by Lemma \ref{lm-zero-spect},
there exists at least one connected closed interval $\mathcal{I}\subset\mathbb{R}$ and a $j\in \{1,2,...,n\}$
such that $\lambda_{j}(\mu)$ for $\mu\in\mathcal{I}$ satisfies \eqref{cond-kT}
and ${\rm Re}\lambda_{j}(\mu)>0$ for each $\mu\in\mathcal{I}$.
This together with (ii) in this theorem yields that for  sufficiently small $\epsilon>0$,
the spectrum of the linear operator $\mathcal{L}$ has a continuous curves lying in the unstable half plane ${\rm Re}\lambda>0$.
Therefore, the proof is now complete.
\end{proof}

\begin{remark}
We point out that the condition in \eqref{cond-kT} is necessary for
the persistence of continuous spectral curves.
Consider a scalar second-order operator
$$\mathcal{L}p:=\partial^{2}_{\xi}p+(a_{0}^{0}+2\epsilon \cos (2 \xi))p=\lambda p.$$
Solving the spectrum of $\mathcal{L}$ is to detect $\lambda\in\mathbb{C}$ such that
\begin{eqnarray}\label{eq-Mathieu}
\partial^{2}_{\xi}p+(a_{0}^{0}-\lambda+\epsilon \cos (2 \xi))p=0
\end{eqnarray}
has a Floquet exponent with zero real part.
It is clear that the spectrum of $\mathcal{L}$ is purely real.
By a direct computation, we have
$\Sigma_{0}^{1}=\{\lambda\in\mathbb{C}: \lambda=a_{0}^{0}-\mu^{2}, \ \ \mu\in\mathbb{R}\}$.
We see that $\Sigma_{0}^{1}$ is a continuous spectral curve.
However, by applying the averaging theory, we can find that
there exist $\lambda$ in a neighborhood of $a_{0}^{0}-k^{2}$ for $k\in \mathbb{Z}$
such that for a sufficiently small $\epsilon>0$,
\eqref{eq-Mathieu} has a unbounded solution (see more details in \cite[p.130]{Hale-80}).
This implies that the  Floquet spectrum curve of $\mathcal{L}$ is broken near $\lambda=a_{0}^{0}-k^{2}$.
\end{remark}

\begin{remark}
By (i) of Theorem \ref{thm-1},
the generalized algebraic multiplicity $m_{ga}(\lambda_{0})$ of $\lambda_{0}$ for the limiting operator
is equal to the number of the eigenvalues of $\mathcal{L}$ perturbed from $\lambda_{0}$ (counted by multiplicity).
It is also interesting to study how to count
the algebraic and geometric multiplicities of each perturbed eigenvalue
$\tilde{\lambda}_{j}(\epsilon)$ ($j=1,2,...,m_{ga}(\lambda_{0})$)
by the limiting operator and it perturbation,
which still remains to be clarified.
\end{remark}

\section{Applications and remarks}
\label{sec:app}

In this section, we apply the main results in the present paper to study the stability of
patterns listed in three examples at the beginning of this paper.

\subsection{Ginzburg-Landau equation}

Consider the real Ginzburg-Landau equation \eqref{GLE} in Example 1.
It is clear that  $(\alpha,\beta)$ satisfies a two component reaction-diffusion equation
\begin{eqnarray}\label{GLE-AB}
\begin{aligned}
\frac{\partial \alpha}{\partial t} &
    =\alpha_{\xi\xi}-\frac{2\omega}{p^{2}}\beta_{\xi}+\frac{2\omega p_{\xi}}{p^{3}}\beta
      +\left(1-\frac{\omega^{2}}{p^{4}}-(\alpha^{2}+\beta^{2})\right)\alpha,
\\
\frac{\partial \beta}{\partial t} &
    =\beta_{\xi\xi}+\frac{2\omega}{p^{2}}\alpha_{\xi}-\frac{2\omega p_{\xi}}{p^{3}}\alpha
      +\left(1-\frac{\omega^{2}}{p^{4}}-(\alpha^{2}+\beta^{2})\right)\beta.
\end{aligned}
\end{eqnarray}
This system has a periodic steady solution $(p(\xi),0)$.
The linearization $\mathcal{L}$ of \eqref{GLE-AB} about $(p,0)$ is in the form \eqref{df-L},
where $D$, $a_{1}$ and $a_{0}$ respectively have the form
\begin{eqnarray} \label{GLE-a-01}
\ \ \ \ \ \ \  D=\left(
\begin{array}{cc}
1 &  0\\
0 &  1
\end{array}
\right), \
a_{1}(\xi,\epsilon)=\left(
\begin{array}{cc}
0 &  -\frac{2\omega}{p^{2}}\\
\frac{2\omega}{p^{2}} &  0
\end{array}
\right), \
a_{0}(\xi,\epsilon)=\left(
\begin{array}{cc}
1-\frac{\omega^{2}}{p^{4}}-3p^{2} & \frac{2\omega p_{\xi}}{p^{3}} \\
-\frac{2\omega p_{\xi}}{p^{3}} &  1-\frac{\omega^{2}}{p^{4}}-p^{2}
\end{array}
\right).
\end{eqnarray}
By applying  Theorem \ref{thm-1} to the real Ginzburg-Landau equation \eqref{GLE},
we obtain the following result.
\begin{proposition} \label{prop-SBL}
For a fixed $\omega\in (0,\sqrt{4/27})$,
the quasiperiodic patterns near the periodic pattern $u(\xi)=p_{1}e^{ikx}$
of the real Ginzburg-Landau equation \eqref{GLE} is spectrally unstable,
where  $k^{2}+p_{1}^{2}=1$ and $kp_{1}^{2}=\omega$.
\end{proposition}
\begin{proof}
Since $a_{1}$ and $a_{0}$ in \eqref{GLE-a-01} satisfy
\begin{eqnarray}\label{eq-limit-GLE}
a_{1}(\xi,\epsilon)
\to a_{1}^{0}=\left(
\begin{array}{cc}
0 &  -\frac{2\omega_{0}}{p^{2}_{1}}\\
\frac{2\omega_{0}}{p^{2}_{1}} &  0
\end{array}
\right), \ \ \
a_{0}(\xi,\epsilon)\to a_{0}^{0}=\left(
\begin{array}{cc}
-2p_{1}^{2} & 0 \\
0 &  0
\end{array}
\right),
\end{eqnarray}
as $p\to p_{1}$,
where we use the equality $\omega^{2}=p_{1}^{4}-p_{1}^{6}$.
Then a direct computation yields that
the corresponding $\Delta_{1}(\lambda,\mu)$ is in the form
\begin{eqnarray*}
\Delta_{1}(\lambda,\mu)
 =\lambda^{2}+2(p_{1}^{2}+\mu^{2})\lambda+\mu^{2}(\mu^{2}+6p_{1}^{2}-4).
\end{eqnarray*}
Since $\omega\in (0,\sqrt{4/27})$, we have $p_{1}<\sqrt{2/3}$.
This implies that $6p_{1}^{2}-4<0$.
Thus for each $\mu$ with $0\leq \mu^{2}\leq 4-6p_{1}^{2}$,
equation $\Delta_{1}(\lambda,\mu)=0$ has a non-negative root $\lambda_{1}(\mu)>0$.
By (ii) in Theorem \ref{thm-1} and
the fact that the spectrum of $\mathcal{L}$ is real (see \cite[p.506]{Doelman-Cardner-Jones-95}),
the spectrum of the quasiperiodic patterns near the periodic pattern $u(\xi)=p_{1}e^{ikx}$
contains  a non-empty interval along the positive real axis.
This finishes the proof.
\end{proof}

\begin{remark}
We point out that the limiting matrix $a_{0}^{0}$ in \eqref{eq-limit-GLE}  has two non-positive eigenvalues.
We can not obtain the instability of quasi-periodic pattern near the periodic pattern $u(\xi)=p_{1}e^{ikx}$
by the eigenvalues of $a_{0}^{0}$.
So we require to analyze the spectral curves of the related limiting operators
and then obtain the instability by (ii) in  Theorem \ref{thm-1}.
\end{remark}

\subsection{Reaction-diffusion systems coupled with ODEs}

Consider the coupled system \eqref{PDE-ODE} in Example 2.
The linearization of \eqref{PDE-ODE} about $(\phi_{\epsilon}(\xi+ct),\psi_{\epsilon}(\xi+ct))$ is in the form \eqref{df-L}.
More precisely, the corresponding  $D$, $a_{1}$ and $a_{0}$ are given by
\begin{eqnarray*}
D=\left(
\begin{array}{cc}
1 &  0\\
0 &  0
\end{array}
\right), \ \
a_{1}(\xi,\epsilon)=\left(
\begin{array}{cc}
-c &  0\\
0 &  -c
\end{array}
\right), \ \
a_{0}(\xi,\epsilon)=\left(
\begin{array}{cc}
f_{u}(\phi_{\epsilon},\psi_{\epsilon},\alpha) & f_{w}(\phi_{\epsilon},\psi_{\epsilon},\alpha) \\
g_{u}(\phi_{\epsilon},\psi_{\epsilon},\alpha) &  g_{w}(\phi_{\epsilon},\psi_{\epsilon},\alpha)
\end{array}
\right),
\end{eqnarray*}
where $\xi\in\mathbb{R}$, $\alpha=\alpha_{0}+O(\epsilon)$ and $c=c_{0}+O(\epsilon)$
for sufficiently small $\epsilon>0$.
Recall that $\Delta_{2}(\lambda,\mu)$ is defined by \eqref{ED-cond-2}.
By a direct computation, we have
\begin{eqnarray*}
\Delta_{2}(\lambda,\mu)
\!\!\!&=&\!\!\!
{\rm det}\left(
\begin{array}{cc}
\lambda+\mu^{2}+{\bf i}c_{0}\mu -f_{u}(0,0,\alpha_{0})  & -f_{w}(0,0,\alpha_{0}) \\
-g_{u}(0,0,\alpha_{0}) & \lambda+{\bf i}c_{0}\mu -g_{w}(0,0,\alpha_{0})
\end{array}
\right)\\
\!\!\!&=&\!\!\! \lambda^{2}+(\mu^{2}+2{\bf i}c_{0}\mu-\varpi_{0}^{2})\lambda-\mu^{2}(c_{0}^{2}+g_{w}(0,0,\alpha_{0}))\\
\!\!\!&=&\!\!\! \lambda(\lambda+\mu^{2}+2{\bf i}c_{0}\mu-\varpi_{0}^{2}).
\end{eqnarray*}
Solving $\Delta_{2}(\lambda,\mu)=0$ yields
\begin{eqnarray*}
\Sigma_{0}^{2}=\left\{\lambda\in\mathbb{C}:\
\lambda_{1}(\mu)\equiv 0,\  \lambda_{2}(\mu)=-\mu^{2}-2{\bf i}c_{0}\mu+\varpi_{0}^{2}, \ \mu\in\mathbb{R}\right\}.
\end{eqnarray*}
The perturbation results for the perturbed periodic wave $(\phi_{\epsilon},\psi_{\epsilon})$
are summarized as follows.

\begin{proposition}\label{FHN}
For a sufficiently small $\epsilon>0$,
the linearization of \eqref{PDE-ODE} about $(\phi_{\epsilon},\psi_{\epsilon})$ satisfies the following statements:
\begin{enumerate}
\item[(i)]
For each $\mu\in\mathbb{R}$,
there exists only one $\lambda_{\epsilon}(\mu)$ perturbed from $\lambda_{j}(\mu)$ for $j=1,2$.

\item[(ii)] For a connected closed $\mathcal{I}\subset \mathbb{R}$ bounded away from $\mu=k\varpi_{0}/2$
for all $k\in\mathbb{Z}$,
there exists a continuous spectral curve perturbed from
$\{\lambda\in \mathbb{C}: \lambda=\lambda_{2}(\mu):=-\mu^{2}-2{\bf i}c_{0}\mu+\varpi_{0}^{2}, \ \ \mu\in\mathcal{I}\}$.

\item[(iii)] The perturbed periodic waves from fold-Hopf bifurcation are spectrally unstable.
\end{enumerate}
\end{proposition}
\begin{proof}
Since $\lambda_{2}(\mu)=0$ has no real roots,
then all generalized algebraic multiplicities  of $m_{ga}(\lambda_{j}(\mu))$ of $\lambda_{j}(\mu)$  are equal to one.
Then the proof for (i) is finished by (i) in Theorem \ref{thm-1}.

Recall that the period of the perturbed periodic wave $(\phi_{\epsilon},\psi_{\epsilon})$ has the expansion
\begin{eqnarray*}
T_{\epsilon}=2\pi/\varpi_{0}+O(\epsilon), \ \ \ 0<\epsilon \ll 1.
\end{eqnarray*}
Then by (ii) in Theorem \ref{thm-1}, the statement in (ii) holds.

Substituting $\mu=0$ into $\lambda_{2}(\mu)$ yields $\lambda_{2}(0)=\varpi_{0}^{2}$.
Then by (i) there exists one $\lambda_{\epsilon}(0)$ perturbed from $\lambda_{2}(0)=\varpi_{0}^{2}$ for sufficiently small $\epsilon>0$.
This implies that  the spectrum of the linearization $\mathcal{L}$ has an element in ${\rm Re}\lambda>0$.
Thus, (iii) is proved. Therefore, the proof is now complete.
\end{proof}

\begin{remark}
This proposition generalizes the previous results in \cite[Theorem 3.8]{Chen-Duan-21}
and obtain the perturbations of the related spectrum and spectral curves.
The arguments in the present paper is not necessarily to check the conditions for relatively bounded perturbation
\cite{Chen-Duan-21,Kato-95},
which is always involved complicated estimates.
With our obtained results, it become easier to prove  \cite[Theorem 1.4]{Alvarez-Plaza-21}, \cite[Theorem 1.3]{Alvarez-Murillo-Plaza-22}
and \cite[Theorem 3.8]{Chen-Duan-21}.
\end{remark}

\subsection{Hyperbolic Burgers-Fisher model}

Consider  the hyperbolic Burgers-Fisher model \eqref{H-model} in Example 3.
The spectral stability of periodic traveling wave $(\phi_{h},\psi_{h})$ with respect to \eqref{H-model}
is determined by a linear operator in the form \eqref{df-L},
where  $D$, $a_{1}$ and $a_{2}$ are respectively given by
\begin{eqnarray*}
D=\left(
\begin{array}{cc}
0 &  0\\
0 &  0
\end{array}
\right), \ \
a_{1}(\xi,\epsilon)=\left(
\begin{array}{cc}
1/2 &  -1\\
-1 &  1/2
\end{array}
\right), \ \
a_{0}(\xi,\epsilon)=\left(
\begin{array}{cc}
g'(\phi_{h}(\xi)) & 0 \\
f'(\phi_{h}(\xi)) &  -1
\end{array}
\right),
\end{eqnarray*}
where $\xi\in\mathbb{R}$ and $h\in(-1/6,0)$.
Without loss of generality, this linearized operator is denoted by $\mathcal{L}^{3}$.

Note that
\begin{eqnarray}\label{eq-limt}
g'(\phi_{h})\to g'(0)=1, \ \ \ f'(\phi_{h})\to f'(0)=0,\ \ \ \ \mbox{ as }\ h\to 0^{-}.
\end{eqnarray}
Then  the related matrix $\Delta_{3}(\lambda,\mu)$ is in the form
\begin{eqnarray*}
\Delta_{3}(\lambda,\mu)
\!\!\!&=&\!\!\!
{\rm det}\left(
\begin{array}{cc}
\lambda-1-{\bf i}\frac{\mu }{2} & {\bf i}\mu \\
{\bf i}\mu & \lambda+1-{\bf i}\frac{\mu }{2}
\end{array}
\right).
\end{eqnarray*}
Solving $\Delta_{3}(\lambda,\mu)=0$ yields
\begin{eqnarray*}
\Sigma_{0}^{3}=\left\{\lambda\in\mathbb{C}: \lambda_{1}(\mu)={\bf i}\frac{\mu }{2}+\sqrt{1-\mu^{2}}, \
               \lambda_{2}(\mu)={\bf i}\frac{\mu }{2}-\sqrt{1-\mu^{2}}, \ \ \mu\in\mathbb{R}\right\}.
\end{eqnarray*}

Concerning  the spectrum of the linearized operator $\mathcal{L}^{3}$,
we have the following results.
\begin{proposition}\label{Prop-H-model}
The linearized operator $\mathcal{L}^{3}$ satisfies the following statements:
\begin{enumerate}
\item[(i)]
There exists only one $\lambda_{\epsilon}(\mu)$ perturbed from each $\lambda_{j}(\mu)$ for $j=1,2$ and $\mu\neq \pm 1$,
and two eigenvalues $\lambda_{\epsilon}^{1,2}(\pm 1)$ perturbed from $\lambda=\pm{\bf i}/2$.

\item[(ii)] For a connected closed $\mathcal{I}\subset \mathbb{R}$ bounded away from $\mu=k/\sqrt{3}$ for all $k\in\mathbb{Z}$
and $\mu=\pm1$, there exists a continuous spectral curve perturbed from
$\{\lambda\in \mathbb{C}: \lambda=\lambda_{1}(\mu)={\bf i}\frac{\mu }{2}\pm\sqrt{1-\mu^{2}}, \ \ \mu\in\mathcal{I}\}$.

\item[(iii)] There exists a sufficiently small $h_{0}>0$ such that periodic traveling waves $(\phi_{h},\psi_{h})$ for $h\in(-h_{0},0)$
are spectrally unstable.
\end{enumerate}
\end{proposition}
\begin{proof}
We begin by verifying that
the linearized operator $\mathcal{L}^{3}$ satisfies the hypotheses {\bf (H1)} and {\bf (H2)} as $h\to 0^{-}$.
Note that \eqref{eq-limt} implies {\bf (H2)}.
Now we only need to check {\bf (H1)}.
For each $h\in(-1/6,0)$,
let $T(h)$ denote the period of $(\phi_{h},\psi_{h})$.
By the change of the variables $u=u$ and $w=u-2v$,
system \eqref{H-2D-model-1} is transformed into
\begin{eqnarray*}
u' =\frac{2}{3}w,\ \ \ \ \
w' =2u^{2}-2u,
\end{eqnarray*}
Then the periodic orbits $\tilde{\Gamma}_{\tilde{h}}$ of the above equation are determined by the following equation
\[
\tilde{h}=\frac{1}{3}w^{2}+u^{2}-\frac{2}{3}u^{3}, \ \ \ \tilde{h}\in (0,1/3).
\]
Let
\begin{eqnarray*}
u=\sqrt{\tilde{h}}\cos\theta (1+\sqrt{3\tilde{h}}\psi), \ \ \ \ w=\sqrt{3\tilde{h}}\sin\theta (1+\sqrt{3\tilde{h}}\psi),
\end{eqnarray*}
where
$
\psi=\frac{1}{3\sqrt{3}}\cos^{3}\theta+O(\sqrt{\tilde{h}})
$
for sufficiently small $h$.
Then we have
\begin{eqnarray*}
T(h)\!\!\!&=&\!\!\!\frac{3}{2}\oint_{\tilde{\Gamma}_{\tilde{h}}}\frac{du}{w}\\
    \!\!\!&=&\!\!\! -\frac{3}{2}\int_{0}^{2\pi}\frac{1}{w}\frac{d u}{d\theta}d\theta\\
    \!\!\!&=&\!\!\! -\frac{3}{2}\int_{0}^{2\pi}
       \frac{1}{\sqrt{3\tilde{h}}\sin\theta (1+\sqrt{3\tilde{h}}\psi)}
       \left(-\sqrt{\tilde{h}}\sin\theta (1+\sqrt{3\tilde{h}}\psi)+\sqrt{3}\tilde{h}\cos\theta \frac{d\psi}{d\theta}\right) d\theta\\
      \!\!\!&=&\!\!\!  \sqrt{3}\pi+O(\sqrt{\tilde{h}}),
\end{eqnarray*}
which yields that
\begin{eqnarray}\label{PF-LIMT}
T(h)\to \sqrt{3}\pi, \ \ \  \mbox{ as }  h \to 0.
\end{eqnarray}
So the linearized operator $\mathcal{L}^{3}$ also satisfies {\bf (H1)}.

Note that ${\bf i}\frac{\mu }{2}+\sqrt{1-\mu^{2}}={\bf i}\frac{\mu }{2}-\sqrt{1-\mu^{2}}$ if and only if $\mu=\pm 1$.
Then by (i) in Theorem \ref{thm-1} we prove the statements in (i).
By \eqref{PF-LIMT} and (ii) in Theorem \ref{thm-1},  we obtain (ii).
The last statement is proved by (iii) in Theorem \ref{thm-1}.
Therefore, the proof is now complete.
\end{proof}

\begin{remark}
Integrable systems always admit periodic solutions, which form a periodic annulus.
See, for instance, the family of periodic traveling waves $(\phi_{h},\psi_{h})$ for the hyperbolic Burgers-Fisher model \eqref{H-model}.
So our obtained results can be also used to deal with  the similar  problems for many other models,
such as the generalized Ginzburg-Landau equations \cite{Duan-Holme-95}.
\end{remark}

\section*{Acknowledgments}

This work was partly supported by the National Natural Science Foundation of China (Grant No. 12101253) and
the Scientific Research Foundation of CCNU (Grant No. 31101222044).
We would like to  thank the editor and the anonymous
referees for carefully reading the manuscript and providing valuable comments,
which are helpful for improving our manuscript.

\section*{CRediT authorship contribution statement}
{\bf Shuang Chen}: Conceptualization, Methodology, Formal analysis,  Writing-review \& editing-original draft.
{\bf Jinqiao Duan}: Conceptualization, Methodology, Writing-review \& editing-original draft.

\section*{Declaration of competing interest}

The authors declare that they have no known competing
financial interests or personal relationships that could have
appeared to influence the work reported in this paper.

\bibliographystyle{amsplain}

\end{document}